\documentclass[a4paper,12pt]{article}

\usepackage{amsthm}
\usepackage[pdftex]{graphicx}
\usepackage{lscape}
\usepackage{bm}
\usepackage[all]{xy}
\usepackage{enumerate}
\usepackage{amsmath,amssymb,amsfonts}
\usepackage{comment}
\usepackage{nccmath}
\usepackage{latexsym,makeidx}
\usepackage[top=30truemm,bottom=30truemm,left=25truemm,right=25truemm]{geometry}

\theoremstyle{plain}
\newtheorem{theorem}{Theorem}[section]
\newtheorem{corollary}[theorem]{Corollary}
\newtheorem{lemma}[theorem]{Lemma}
\newtheorem{conjecture}[theorem]{Conjecture}
\newtheorem{proposition}[theorem]{Proposition}

\theoremstyle{definition}
\newtheorem{definition}[theorem]{Definition}
\newtheorem{remark}[theorem]{Remark}
\newtheorem{example}[theorem]{Example}

\newcommand{\vin}{\rotatebox{90}{$\in$}}

\newcommand{\map}[3]{{#1}:{#2}\longrightarrow{#3}}

\newcommand{\C}[1]{[c_0]_{#1}[c_1]_{#1}\cdots [c_{\ell}]_{#1}}
\newcommand{\x}{x}
\renewcommand{\t}{t}
\newcommand{\R}{\mathrm{R}}
\newcommand{\A}{\mathrm{A}}
\newcommand{\Eh}{\mathrm{L}}
\newcommand{\n}{n}
\newcommand{\Sh}{\mathrm{S}}
\newcommand{\Shp}[1]{\mathrm{S}_{#1}}
\newcommand{\degree}{\ell}
\renewcommand{\gcd}{\mathrm{gcd}}
\newcommand{\mult}{\ell}
\newcommand{\geneshift}{\mathrm{L}_{A_{\ell}}}
\newcommand{\Ehp}{\rho}
\newcommand{\avEh}[2]{\tilde{\mathrm{L}}_{#1}^{#2}}
\newcommand{\nn}{n}
\newcommand{\divcoxc}[1]{\tilde{n}_{c_{#1}}}
\newcommand{\Ehf}[2]{\mathrm{L}_{#1}^{#2}}
\newcommand{\Set}[2]{\{{#1}\ |\ {#2}\}}
\newcommand{\difi}{\hat{c}}
\newcommand{\ndifi}{\hat{\ell}}
\newcommand{\adegree}{\ell}
\newcommand{\pf}{p}
\newcommand{\m}{m}
\newcommand{\coef}{a}
\newcommand{\ascl}{a}
\newcommand{\asc}{\mathrm{asc}}

\renewcommand{\map}[3]{{#1}:{#2}\rightarrow{#3}}

\newcommand{\coxc}{c}
\newcommand{\coxn}{h}
\newcommand{\cyc}[2]{[#1]_{#2}}
\newcommand{\p}{n}
\newcommand{\pe}{p}
\newcommand{\sg}{s}
\newcommand{\ig}{i}

\newcommand{\rg}{r}
\newcommand{\dg}{d}
\newcommand{\qg}{q}

\newcommand{\divisors}[1]{X_{#1}}
\newcommand{\invgcd}{\check{d}}
\newcommand{\mug}{\mu}

\newcommand{\rad}{\mathrm{rad}}
\renewcommand{\Re}{\operatorname{Re}}
\renewcommand{\dim}{\operatorname{dim}}
\renewcommand{\ker}{\operatorname{ker}}

\newcommand{\mapel}[5]{
			\begin{array}{ccc}
				{#1}:{#2} & \longrightarrow & {#3}\\
				\vin & & \vin \\[-4pt]
				{#4} & \longmapsto & {#5}\\
			\end{array}}

\begin{document}
\title{Postnikov--Stanley Linial arrangement conjecture}
\author{Shigetaro Tamura \thanks{Department of Mathematics, Faculty of Science,
Hokkaido University, Kita 10, Nishi 8, Kita-ku, Sapporo 060-0810, JAPAN.
E-mail: tamurashigetaro@gmail.com}}
\maketitle
\begin{abstract}
A characteristic polynomial is an important invariant in the field of hyperplane arrangement. For the Linial arrangement of any irreducible root system, Postnikov and Stanley conjectured that all roots of the characteristic polynomial have the same real part. In relation to this conjecture, Yoshinaga obtained an explicit relationship between the characteristic quasi-polynomial and the Ehrhart quasi-polynomial for the fundamental alcove. In this paper, we calculate Yoshinaga's explicit formula through the decomposition of the Ehrhart quasi-polynomial into several quasi-polynomials and a modified shift operator, and obtain new formulas for the characteristic quasi-polynomial of the Linial arrangement. In particular, when the parameter of the Linial arrangement is relatively prime to the period of the Ehrhart quasi-polynomial, we prove the Postnikov--Stanley Linial arrangement conjecture. This generalizes some of the results for the root systems of classical types that have been proved by Postnikov--Stanley and Athanasiadis. For other cases, we verify this conjecture for exceptional root systems using a computational approach.

\end{abstract}
\tableofcontents
\section{Introduction}
Let $\mathcal{A}$ be a hyperplane arrangement, that is, a finite collection of affine hyperplanes in a vector space $V$. One of the most important invariants of $\mathcal{A}$ is the characteristic polynomial $\chi(\mathcal{A},t)$. Let $\Phi$ be an irreducible root system with the Coxeter number $\coxn$. Let $a,b\in \mathbb{Z}$ be integers with $a\leqq b$. Let us denote by $\mathcal{A}_{\Phi}^{[a,b]}$ the truncated affine Weyl arrangement. In particular, $\mathcal{A}_{\Phi}^{[1,\nn]}$ is called the Linial arrangement. Postnikov and Stanley \cite{Postnikov-Stanley} conjectured that
every root $z \in \mathbb{C}$ of the equation $\chi(\mathcal{A}_{\Phi}^{[1,\nn]},t)=0$ satisfies $\Re z=\frac{\nn h}{2}$ (see \S \ref{section:Conjecture} for details).\par
Postnikov and Stanley proved this conjecture for $\Phi=A_{\degree}$ \cite{Postnikov-Stanley}. Subsequently, Athanasiadis gave proofs for $\Phi=A_{\degree}$, $B_{\degree}$, $C_{\degree}$, and $D_{\degree}$ using a combinatorial method \cite{Athanasiadis}. Yoshinaga approached the conjecture through the characteristic quasi-polynomial $\chi_{quasi}(\mathcal{A}_{\Phi}^{[1,\nn]},t)$, which was introduced by Kamiya et al.~\cite{Kamiya-Takemura-Terao_0}. The characteristic quasi-polynomial $\chi_{quasi}(\mathcal{A}_{\Phi}^{[1,\nn]},t)$ has the important property that when $t$ is relatively prime to the period of $\chi_{quasi}(\mathcal{A}_{\Phi}^{[1,\nn]},t)$, the formula $\chi_{quasi}(\mathcal{A}_{\Phi}^{[1,\nn]},t)=\chi(\mathcal{A}_{\Phi}^{[1,\nn]},t)$ holds \cite[Theorem 2.1]{Athanasiadis}. Yoshinaga has proved the following formula \cite{Yoshinaga_1} (see Theorem \ref{characteristic_quasi_poly}).
\begin{equation}\label{intro_ch}
\chi_{quasi}(\mathcal{A}_{\Phi}^{[1,\nn]},t)=\R_{\Phi}(\Sh^{\nn+1})\Eh_{\Phi}(t),
\end{equation}
where $\Sh$ is the shift operator for the variable $t$ (see \S \ref{Shift congruences}), $\Eh_{\Phi}(t)$ is the Ehrhart quasi-polynomial for the closed fundamental alcove of type $\Phi$ (see \S \ref{section:Eh_quasi}), and $\R_{\Phi}(t)$ is the generalized Eulerian polynomial of type $\Phi$, which was introduced by Lam and Postnikov \cite{Lam-Postnikov} (see \S \ref{section:generalized Eulerian}). By using this formula, Yoshinaga verified several cases of the conjecture (see \S \ref{section:Conjecture}).

\subsection{Main results}
Let $\Ehp$ be the period of the characteristic quasi-polynomial $\chi_{quasi}(\mathcal{A}_{\Phi}^{[1,\nn]},t)$. Let $\m$ be an integer with $\nn+1=\m\cdot \gcd(\nn+1,\Ehp)$. Let $\coxc_0,\cdots, \coxc_{\degree}$ be integers that are coefficients of each simple root when the highest root is expressed as a linear combination of simple roots in an irreducible root system $\Phi$ of rank $\degree$ (see \S \ref{root system}). By calculating the right-hand side of (\ref{intro_ch}), we prove the formula 
\begin{equation}\label{intro_}
\chi_{quasi}(\mathcal{A}_{\Phi}^{[1,\nn]},t)=(\prod_{j=0}^{\degree}\frac{1}{\m}[\m]_{\Sh^{\coxc_j\cdot \gcd(\nn+1,\Ehp)}}) \chi _{quasi}(\mathcal{A}_{\Phi}^{[1,\gcd(\nn+1,\Ehp)-1]},t),
\end{equation}
where $\cyc{m}{t}=\frac{1-t^{m}}{1-t}=1+t+\cdots+t^{m-1}$ (see Theorem \ref{corollary_1}). Furthermore, the characteristic quasi-polynomial $\chi_{quasi}(\mathcal{A}_{\Phi}^{[1,\nn]},t)$ has the period $\gcd(\nn+1,\Ehp)$. In particular, when the parameter $\nn+1$ is relatively prime to the period $\Ehp$ of the Ehrhart quasi-polynomial $\Eh_{\Phi}(t)$, that is, $\gcd(\nn+1,\Ehp)=1$, we have
\begin{equation}\label{intro_gcd}
\chi_{quasi}(\mathcal{A}_{\Phi}^{[1,\nn]},t)=(\prod_{j=0}^{\degree}\frac{1}{\nn+1}[\nn+1]_{\Sh^{\coxc_j}})t^{\degree}
\end{equation}
from (\ref{intro_}) (see Theorem \ref{gcd_prime}). In this case, from (\ref{intro_gcd}) and the technique used by Postnikov and Stanley in \cite{Postnikov-Stanley} (see Lemma \ref{Postnikov-Stanley's lemma}), we see that the conjecture holds. In addition, we prove the formula for the characteristic polynomial
\begin{equation}\label{intro_rad}
\chi(\mathcal{A}_{\Phi}^{[1,\gcd(\nn+1,\Ehp)-1]},t)=(\prod_{j=0}^{\degree}\frac{1}{\eta}[\eta]_{\Sh^{\coxc_j\cdot \gcd(\nn+1,\rad(\Ehp))}}) \chi(\mathcal{A}_{\Phi}^{[1,\gcd(\nn+1,\rad(\Ehp))-1]},t),
\end{equation}
where $\eta=\frac{\gcd(\nn+1,\Ehp)}{\gcd(\nn+1,\rad(\Ehp))}$ (see Theorem \ref{Ch_rad}). From (\ref{intro_}) and (\ref{intro_rad}), if all roots of the characteristic polynomial $\chi(\mathcal{A}_{\Phi}^{[1,\gcd(\nn+1,\rad(\Ehp))-1]},t)$ have the same real part $\frac{(\gcd(\nn+1,\rad(\Ehp))-1)\coxn}{2}$, then the same method as for (\ref{intro_gcd}) can be used to show that $\chi(\mathcal{A}_{\Phi}^{[1,\nn]},t)$ satisfies the conjecture. We can check the conjecture for $\Phi \in \{E_6,E_7,E_8,F_4\}$ by computing the real part of all roots of $\chi_(\mathcal{A}_{\Phi}^{[1,\gcd(\nn+1,\Ehp)-1]},t)$ using a computational approach.

\subsection{Outline of the proof}
To prove the conjecture, we transform the right-hand side of (\ref{intro_ch}) into a suitable form. One of the difficulties in this transformation is that the shift operator $\Sh$ acts on a quasi-polynomial, not a polynomial. To overcome this difficulty, we introduce the operator $\overline{\Sh}$, which acts on a constituent of a quasi-polynomial (Definition \ref{shift_bar}). Additionally, we define a quasi-polynomial $\tilde{f}^{i}(t)$ from a quasi-polynomial $f(t)$ (Definition \ref{quasi_av}). The quasi-polynomial $\tilde{f}^{i}(t)$ is like an average of the constituents of the quasi-polynomial $f(t)$, and its minimal period is a divisor of the integer $i$. Using a generalization of Lemma 2.2 in \cite{Athanasiadis} (Lemma \ref{Athanasiadis's lemma_2}), for a quasi-polynomial $f(t)$ of degree $\degree$ and period $\Ehp$, we obtain the formula
\begin{equation}\label{intro_shift_bar}
[c]_{\Sh^m}^{\degree+1}g(\Sh^m)f(t)=[c]_{\overline{\Sh}^m}^{\degree+1}g(\overline{\Sh}^m)\tilde{f}^{\mathrm{gcd}(m,\Ehp)}(t),
\end{equation}
where $g(\Sh)$ is the substituted shift operator $\Sh$ for a polynomial $g(t)$ (Proposition \ref{averaging}). 

The Ehrhart quasi-polynomial $\Eh_{\Phi}(t)$ 
decomposes into several quasi-polynomials that have a degree and period that is less than or equal to its own degree and period:
\begin{equation}\label{intro_Eh}
\Eh_{\Phi}(t)=\sum_{k\in \{\difi_0,\cdots, \difi_{\ndifi}\}}\Ehf{k}{(\degree_{k})}(t),
\end{equation}
where $\difi_0,\cdots,\difi_{\ndifi}$ are all the different integers in $\coxc_0,\cdots, \coxc_{\degree}$, $\adegree_{\difi_k}+1$ is the number of multiples of $\difi_k$ in $\coxc_0,\cdots, \coxc_{\degree}$ (see \S \ref{section:Eh_quasi}), and $\Ehf{k}{(\adegree_{k})}(t)$ is a quasi-polynomial of degree $\adegree_{k}$ with period $k$ (Proposition \ref{Eh_deco}). This decomposition is well matched with the following decomposition of generalized Eulerian polynomials, which was proved in \cite{Lam-Postnikov}.
\begin{equation}\label{intro_Eu}
\R_{\Phi}(t)=\cyc{\coxc_0}{t}\cyc{\coxc_1}{t}\cdots\cyc{\coxc_{\degree}}{t}\A_{\degree}(t),
\end{equation}
where $\A_{\degree}(t)$ is the Eulerian polynomial (Theorem \ref{Lam-Postnikov}). The right-hand side of (\ref{intro_Eu}) has the divisor $\cyc{\difi_k}{t}^{\adegree_{\difi_k}+1}$. Hence, we can apply (\ref{intro_shift_bar}) to each $\Ehf{\difi_k}{(\adegree_{\difi_k})}(t)$ of (\ref{intro_Eh}). From the above argument, we have the formula
\begin{equation}\label{intro_R}
\R_{\Phi}(\Sh^{\nn+1})\Eh_{\Phi}(t)=\R_{\Phi}(\overline{\Sh}^{\nn+1})\tilde{\Eh}^{\gcd(\nn+1,\Ehp)}_{\Phi}(t),
\end{equation}
(see Theorem \ref{main theorem}).
We can think of the operator $\R_{\Phi}(\overline{\Sh}^{\nn+1})$ as acting on a polynomial, or more precisely, on a constituent of the quasi-polynomial $\tilde{\Eh}^{\gcd(\nn+1,\Ehp)}_{\Phi}(t)$. Thus, we can easily calculate the right-hand side of (\ref{intro_R}) and prove (\ref{intro_}).\par
The remainder of this paper is organized as follows. Section \ref{section:Pre} contains some preliminaries required to prove the main results. First, we prove a generalization of Athanasiadis' Lemma \cite{Athanasiadis} in \S \ref{Shift congruences}. In \S \ref{section:Pre_quasi}, we introduce the operator $\overline{\Sh}$ and a quasi-polynomial $\tilde{f}^{i}(t)$, and prove (\ref{intro_shift_bar}). In \S \ref{sec:Pre_deco}, we prove that a decomposition of a quasi-polynomial holds using a generating function. In \S \ref{root system}, \S \ref{section:Eh_quasi}, and \S \ref{section:generalized Eulerian}, we prepare several concepts required to explain Yoshinaga's results \cite{Yoshinaga_1} for the characteristic quasi-polynomial $\chi_{quasi}(\mathcal{A}_{\Phi}^{[1,\nn]},t)$, which is explained together with the Postnikov--Stanley Linial arrangement conjecture in \S \ref{section:Conjecture}. The explanations in \S \ref{section:shift operator}, \S \ref{root system}, \S \ref{section:Eh_quasi}, \S \ref{section:generalized Eulerian}, and \S \ref{section:Conjecture} are based on \cite{Yoshinaga_1}. We prove (\ref{intro_}), (\ref{intro_gcd}), and (\ref{intro_rad}) in \S \ref{section:main_formula}. We present a table of the characteristic polynomial $\chi(\mathcal{A}_{\Phi}^{[1,\nn]},t)$ and the real part of all of its roots for $\Phi \in \{E_6,E_7,E_8,F_4\}$ in \S \ref{section:main_check}.

\section{Preliminaries}\label{section:Pre}
\subsection{Shift operator and congruence}\label{Shift congruences}
\subsubsection{Shift operator}\label{section:shift operator}
Let $f:\mathbb{Z}\rightarrow\mathbb{C}$ be a partial function, that is, a function defined on a subset of $\mathbb{Z}$. Define the action of the shift operator by
\begin{equation}
\Sh f(t)=f(t-1).	
\end{equation}
More generally, for a polynomial $g(\Sh)=\sum_{k}a_k\Sh^k$ in $\Sh$, the action is defined by
\begin{equation}
g(\Sh)f(t)=\sum_{k}a_kf(t-k).	
\end{equation}
\begin{proposition}\label{cong}(\cite{Yoshinaga_1}, Proposition 2.8)
Let $g(\Sh) \in \mathbb{C}[\Sh]$ and $f(t) \in \mathbb{C}[t]$. Suppose $\deg f=\degree$. Then, $g(\Sh)f(t)=0$ if and only if $(1-\Sh)^{\degree+1}$ divides $g(\Sh)$.
\end{proposition}
\begin{remark}
Note that, because $(1-\Sh)f(t)=f(t)-f(t-1)$ is the difference operator, $\deg(1-\Sh)f=\deg f-1$. Hence, inductively, $(1-\Sh)^{\deg f +1}f(t)=0$. Proposition \ref{Shift congruences} implies that if polynomials $g_1(\Sh)$ and $g_2(\Sh)$ satisfy the congruence
\begin{equation}\label{congruence}
g_1(t) \equiv g_2(t) \bmod (1-t)^{\degree+1},
\end{equation}
then for any polynomial $f(t)$ of degree less than or equal to $\degree$,
\begin{equation}\label{shift_equation}
g_1(\Sh)f(t)=g_2(\Sh)f(t),
\end{equation}
since $(1-\Sh)^{\degree+1}f(t)=0$.
Conversely, when $g_1(\Sh)f(t)= g_2(\Sh)f(t)$ for a polynomial $f(t)$ of degree $\degree$, (\ref{congruence}) holds.
\end{remark}

\subsubsection{Congruence}
Lemmas \ref{Athanasiadis's lemma_3/2} and \ref{Athanasiadis's lemma_2} are generalizations of Lemma 2.2 in \cite{Athanasiadis}. The proofs of these lemmas are very similar to the proof given by Athanasiadis \cite{Athanasiadis}. Let $\cyc{c}{t}:=\frac{1-t^c}{1-t}=1+t+\cdots+t^{c-1}$, where $c$ is a non-negative integer. 

\begin{lemma}\label{Athanasiadis's lemma_3/2}
If $g(t)=\sum_{k}\coef_kt^{k}$ is a polynomial and $\n$ is a positive integer, then $g(t)$ can be divided by $\cyc{\n}{t}^{\degree+1}$ if and only if the following formulas hold for any integer $r \in \{0,1,\cdots,\degree\}$.
\begin{equation}
\sum_{k \equiv 0 \bmod n}\coef_k k^r= \sum_{k \equiv 1 \bmod \n}\coef_k k^r= \cdots=\sum_{k \equiv \n-1 \bmod \n}\coef_k k^r.  
\end{equation}
\end{lemma}
\begin{proof}
Let $\omega:=\mathrm{e}^{\frac{2 \pi \sqrt{-1}}{n} }$. First, suppose that $g(t)=\cyc{\n}{t}^{\degree+1}h(t)$, where $h(t)$ is a polynomial.
\[
\sum_{k}\coef_k k^r t^k=\biggl(t\frac{d}{dt} \biggr)^r\cyc{\n}{t}^{\degree+1}h(t).
\]
Since $r \leqq \degree$ and using Leibniz's rule,
\[
\sum_{k}\coef_k k^r \omega^k=0.
\]		
From $\omega^n=1$,
\[
(\sum_{k \equiv 0 \bmod \n}\coef_k k^r) 
+(\sum_{k \equiv 1 \bmod \n}\coef_k k^r \omega)
+\cdots 
+(\sum_{k \equiv \n-1 \bmod \n}\coef_k k^r \omega^{\n-1})=0.
\]
Let $s^{(r)}_{i}:=\sum_{k \equiv i \bmod \n}\coef_k k^r$. Then,
\[s^{(r)}_0+ s^{(r)}_1\omega+\cdots+ s^{(r)}_{\n-1}\omega^{\n-1}=0.\]
Since $\omega^2,\cdots,\omega^{\n-1}$ are also $\n$-th roots of unity, we obtain the formulas
\begin{equation}\label{n-th root}
\begin{array}{llll}
s^{(r)}_0+ s^{(r)}_1\omega &+\cdots+s^{(r)}_{\n-1}\omega^{\n-1}&=0,\\
s^{(r)}_0+ s^{(r)}_1\omega^2 &+\cdots+ s^{(r)}_{\n-1}\omega^{2(\n-1)}&=0,\\
\quad \vdots\\
s^{(r)}_0+ s^{(r)}_1\omega^{\n-1} &+\cdots+ s^{(r)}_{\n-1}\omega^{(\n-1)^2}&=0.
\end{array}
\end{equation}
Let $s:=(s^{(r)}_0, s^{(r)}_1,\cdots, s^{(r)}_{\n-1})^T$. Let us define an $(\n-1) \times \n$ matrix $W$ as\[
W := \left(
\begin{array}{llll}
1 & \omega & \ldots & \omega^{\n-1} \\
1 & \omega^{2} & \ldots & \omega^{2(\n-1)} \\
\vdots & \vdots & \ddots & \vdots \\
1 & \omega^{\n-1} & \ldots & \omega^{(\n-1)^2}
\end{array}\right).
\]
We rewrite (\ref{n-th root}) as $Ws=0$. Since $\omega$ is primitive, we have that $\dim(\ker W)=1$ from Vandermonde's determinant. By $(1,\cdots,1) \in \ker W$, we obtain the formula $s_0^{(r)}=\cdots= s_{\n-1}^{(r)}$. Conversely, suppose that $s_0^{(r)}=\cdots= s_{\n-1}^{(r)}$ for any $r\in\{0,1,\cdots,\degree\}$.  Then, for any $r\in\{0,1,\cdots,\degree\}$ and $m\in\{0,\cdots,\n-1\}$,
\begin{equation}
\biggl(t\frac{d}{dt} \biggr)^rg(t)\Bigl|_{t=\omega}=\sum_{k}\coef_k k^{r}(\omega^{m})^k=0.	
\end{equation}
By induction on the parameter $\degree$, we find that the polynomial $\cyc{\n}{t}^{\degree+1}$ divides $g(t)$.
\end{proof}
We prove the following lemma using Lemma \ref{Athanasiadis's lemma_3/2}.
\begin{lemma} \label{Athanasiadis's lemma_2}
If $g(t)=\sum_{k}\coef_kt^{k}$ is a polynomial and $\n$ is a positive integer, then a polynomial $g(t)$ can be divided by $\cyc{\n}{t}^{\degree+1}$ if and only if the following formulas hold. 
\begin{equation}\label{cyclotomic congruence}
\frac{1}{n}g(t) \equiv \sum_{k\equiv 0 \bmod \n}\coef_kt^k \equiv \cdots \equiv \sum_{k\equiv \n-1 \bmod \n}\coef_kt^k \ \bmod (1-t)^{\degree+1}.
\end{equation}
\end{lemma}
\begin{proof}
First, suppose that $g(t)=\cyc{\n}{t}^{\degree+1}h(t)$, where $h(t)$ is a polynomial. We prove the formula using the shift operator action on $f(t)= t^{\degree}$. For any $j \in \{0,1,\cdots,\n-1\}$,
\[
\begin{split}
\sum_{k\equiv j \bmod n}\coef_k\Sh^k t^{\degree}
&= \sum_{k\equiv j \bmod \n}\coef_k(t-k)^{\degree}\\
&= \sum_{k\equiv j \bmod \n}\coef_k \sum_{r=0}^{\degree} \binom{\degree}{r} (-1)^{r}k^{r}t^{\degree-r}\\
&= \sum_{r=0}^{\degree}\binom{\degree}{r}(-1)^{r} (\sum_{k\equiv j \bmod \n}\coef_kk^{r})t^{\degree-r}.
\end{split}
\]		
By Lemma \ref{Athanasiadis's lemma_3/2},
\[
\sum_{k \equiv 0 \bmod \n}\coef_k \Sh^k t^{\degree} = \sum_{k \equiv 1 \bmod \n}\coef_k \Sh^k t^{\degree} = \cdots=\sum_{k \equiv \n-1 \bmod \n}\coef_k \Sh^k t^{\degree}.
\]		
Thus, for any $j \in \{1,\cdots,n\}$,
\[
\frac{1}{\n}g(\Sh)t^{\degree} = \sum_{k\equiv j \bmod \n}\coef_k\Sh^{k}t^{\degree}.
\]
By Proposition \ref{Shift congruences},
\[
\frac{1}{\n}g(t) \equiv \sum_{k\equiv j \bmod \n}\coef_kt^k \ \bmod (1-t)^{\degree+1}.
\]
The converse is proved by following the above proof  in reverse.
\end{proof}
\subsection{Quasi-polynomial}\label{section:Pre_quasi}
A function $f:\mathbb{Z}\rightarrow \mathbb{C}$ is called a quasi-polynomial if there exists a positive integer $\p>0$ and polynomials $f_1(t),\cdots,f_{\p}(t) \in \mathbb{C}[t]$ such that
\begin{equation}
f(t) = \left\{
\begin{array}{ll}
f_1(t), & t\equiv1 \bmod \p,\\
f_2(t), & t\equiv2 \bmod \p,\\
\quad \vdots\\
f_{\p-1}(t), & t\equiv \p-1 \bmod \p,\\
f_{\p}(t), & t\equiv 0 \bmod \p.
\end{array}\right.
\end{equation}
Such a $\p$ is called the period of the quasi-polynomial $f(t)$. The minimum of the period of $f(t)$ is called the minimal period. The polynomials $f_1(t), \cdots, f_{\p}(t)$ are the constituents of $f(t)$. We define $\deg f:=\underset{1 \leqq i \leqq \p}{\max}\deg f_i$ as the degree of a quasi-polynomial $f(t)$. Moreover, if $f_r(t)=f_{\mathrm{gcd}(r,n)}(t)$ for any $r \in \{1,\cdots,\p\}$, then we say that the quasi-polynomial $f(t)$ has the gcd-property. 
\begin{remark}
We can express a quasi-polynomial as
\[
f(t)=\pf^{\degree}(t)t^{\degree}+\pf^{\degree-1}(t)t^{\degree-1}+\cdots +\pf^{0}(t),	
\]
where $\pf^{\degree}(t),\cdots, \pf^{0}(t)$ are periodic functions. The minimal period of a quasi-polynomial $f(t)$ is the least common multiple of the periods of the periodic functions $\pf^{\degree}(t),\cdots, \pf^{0}(t)$.
\end{remark}
\begin{definition}\label{quasi_av}
Let $f(t)$ be a quasi-polynomial with minimal period $\p$. Let $s$ be a positive integer.\\
\[
\begin{split}
f(t) = \left\{
\begin{array}{ll}
f_1(t), & t\equiv1 \bmod s\p,\\
f_2(t), & t\equiv2 \bmod s\p,\\
\quad \vdots\\
f_{s\p-1}(t), & t\equiv s\p-1 \bmod s\p,\\
f_{s\p}(t), & t\equiv 0 \bmod s\p.
\end{array}\right.
\end{split}
\]
We define the action of the symmetric group $\mathfrak{S}_{s\p}$ on a quasi-polynomial as follows.
\[		
\begin{split}
f^{\sigma}(t) := \left\{
\begin{array}{ll}
f_{\sigma^{-1}(1)}(t), & t\equiv1 \bmod s\p,\\
f_{\sigma^{-1}(2)}(t), & t\equiv2 \bmod s\p,\\
\quad \vdots\\
f_{\sigma^{-1}(s\p-1)}(t), & t\equiv s\p-1 \bmod s\p,\\
f_{\sigma^{-1}(s\p)}(t), & t\equiv 0 \bmod s\p,
\end{array}\right.
\end{split}
\]
where $\sigma \in \mathfrak{S}_{s\p}$. Let $\sigma_{s\p}$ be the cyclic permutation $(1,2,\cdots,s\p)\in \mathfrak{S}_{s\p}$. For any positive integer $s$, we have $f^{\sigma_{s\p}}(t)=f^{\sigma_{\p}}(t)$. In other words, the action of the cyclic permutation $\sigma_{s\p}=(1,\cdots,sn)$ on $f(t)$ does not depend on $s$. From now on, we denote a cyclic permutation $(1,\cdots,\p)$ by $\sigma$, where $\p$ takes the minimal period of a quasi-polynomial on which $\sigma$ acts in each case. Let $k$ be an integer. Define the following quasi-polynomial for $k$:
\begin{equation}
\tilde{f}^k(t):=\frac{f(t)+f^{\sigma^{k}}(t)+ f^{\sigma^{2k}}(t)+ \cdots +f^{\sigma^{(\p-1)k}}(t)}{\p}.
\end{equation}
\end{definition}
\begin{remark}\label{remark_tilde}
\begin{enumerate}[(1)]
\item Let $\p$ be a period of $f(t)$. Let $k$ be a divisor of $\p$ and $m:=\frac{\p}{k}$.
The quasi-polynomial $\tilde{f}^{k}(t)$ has the period $k$.
\[
\begin{split}
\tilde{f}^{k}(t)=\left\{
\begin{array}{ll}
\frac{f_1(t)+f_{k+1}(t)+f_{2k+1}(t)+\cdots+f_{(m-1)k+1}(t)}{m}, & t \equiv 1 \bmod k,\\
\frac{f_2(t)+f_{k+2}(t)+f_{2 k+2}(t)+\cdots+f_{(m-1)k+2}(t)}{m}, & t \equiv 2 \bmod k,\\
\quad \vdots\\
\frac{f_{k-1}(t)+f_{2k-1}(t)+f_{3k-1}(t)+\cdots+f_{mk-1}(t)}{m}, & t \equiv k-1 \bmod k,\\	
\frac{f_{k}(t)+f_{2k}(t)+f_{3k}(t)+\cdots+f_{mk}(t)}{m}, & t \equiv 0 \bmod k.\\\end{array}\right.
\end{split}
\]			
\item\label{remark_tilde_2} When a quasi-polynomial $f(t)$ has a period $\p$, we have that $\tilde{f}^k(t)=\tilde{f}^{k+\p}(t)$.
\end{enumerate}
\end{remark}
\begin{lemma}\label{sigma_linear}
Let $f(t)$, $g(t)$, and $h(t)$ be quasi-polynomials such that $f(t)=g(t)+h(t)$ holds. Then, $f^{\sigma}(t)=g^{\sigma}(t)+h^{\sigma}(t)$, that is, the action of the cyclic permutation $\sigma$ is linear.
\end{lemma}
\begin{proof}
Let $\p$ be the minimal period of $f(t)$. Let $s\p$ be the least common multiple of the minimal periods of $g(t)$ and $h(t)$. Let $g_j(t)$ and $h_j(t)$ be constituents of $g(t)$ and $h(t)$ for $t\equiv j \bmod s\p$. Let $\sigma_{s\p}:=(1,\cdots,s\p)$. Note that we use the notation $\sigma$ as the cyclic permutation for the minimal period of a quasi-polynomial on which $\sigma$ acts, and we have
\[
\begin{split}
f^{\sigma}(t)=f^{\sigma_{s\p}}(t)&=\left\{
\begin{array}{ll}
g_{\sigma_{s\p}^{-1}(1)}(t)+h_{\sigma_{s\p}^{-1}(1)}(t), & t\equiv1 \bmod s\p,\\
g_{\sigma_{s\p}^{-1}(2)}(t)+h_{\sigma_{s\p}^{-1}(2)}(t), & t\equiv2 \bmod s\p,\\
\quad \vdots\\
g_{\sigma_{s\p}^{-1}(s\p-1)}(t)+h_{\sigma_{s\p}^{-1}(s\p-1)}(t), & t\equiv s\p-1 \bmod s\p,\\
g_{\sigma_{s\p}^{-1}(s\p)}(t)+h_{\sigma_{s\p}^{-1}(s\p)}(t), & t\equiv 0 \bmod s\p
\end{array}\right.\\
&=g^{\sigma_{s\p}}(t)+h^{\sigma_{s\p}}(t)\\
&=g^{\sigma}(t)+h^{\sigma}(t).
\end{split}
\]
\end{proof}

\begin{lemma}\label{tilde_linear}
Let $f(t)$, $g(t)$, and $h(t)$ be quasi-polynomials such that $f(t)=g(t)+h(t)$ holds. Let $k$ be an integer. Then, $\tilde{f}^{k}(t)=\tilde{g}^{k}(t)+\tilde{h}^{k}(t)$.
\end{lemma}
\begin{proof}
Let $\p_0,\p_1,\p_2$ be the minimal period of each $f(t),g(t),h(t)$. Note that $\p_1\p_2$ is a multiple of $\p_0$. Then, by Lemma \ref{sigma_linear}, Remark \ref{remark_tilde} (\ref{remark_tilde_2}), 
\[
\begin{split}
\tilde{f}^{k}(t)&=\frac{f(t)+f^{\sigma^{k}}(t)+\cdots+f^{\sigma^{(\p_0-1)k}}}{\p_0}\\
&=\frac{f(t)+f^{\sigma^{k}}(t)+\cdots+f^{\sigma^{(\p_1\p_2-1)k}}}{\p_1\p_2}\\
&=\frac{\bigl(g(t)+h(t)\bigr)+\bigl(g^{\sigma^{k}}(t)+ h^{\sigma^{k}}(t)\bigr)+\cdots+ \bigl(g^{\sigma^{(\p_1\p_2-1)k}}(t)+ h^{\sigma^{(\p_1\p_2-1)k}}(t)\bigr)}{\p_1\p_2}\\
&=\frac{g(t)+g^{\sigma^{k}}(t)+\cdots+ g^{\sigma^{(\p_1\p_2-1)k}}(t)}{\p_1\p_2}+ \frac{h(t)+h^{\sigma^{k}}(t)+\cdots+ h^{\sigma^{(\p_1\p_2-1)k}}(t)}{\p_1\p_2}\\
&=\frac{\p_2(g(t)+g^{\sigma^{k}}(t)+\cdots+ g^{\sigma^{(\p_1-1)k}}(t))}{\p_1\p_2}+ \frac{\p_1(h(t)+h^{\sigma^{k}}(t)+\cdots+ h^{\sigma^{(\p_2-1)k}}(t))}{\p_1\p_2}\\
&=\tilde{g}^{k}(t)+\tilde{h}^{k}(t).
\end{split}
\]
\end{proof}
\begin{proposition}\label{tilde-gcd}
Let $f(t)$ be a quasi-polynomial with period $\p$. Let $k$ be an integer.
\begin{equation}
\tilde{f}^{k}(t)=\tilde{f}^{\gcd(k,\p)}(t).
\end{equation}
In particular, the quasi-polynomial $\tilde{f}^{k}(t)$ has the period $\gcd(k,\p)$. 
\end{proposition}
\begin{proof}	
Let $[b]:=b+\p\mathbb{Z} \in \mathbb{Z}/n \mathbb{Z}$. We will prove that
\begin{equation}\label{modn}
\{[k],[2k],\cdots,[(\p-1)k] \} = \{[\mathrm{gcd}(k,\p)],[2\mathrm{gcd}(k,\p)],\cdots,[(\p-1)\mathrm{gcd}(k,\p)]\}.
\end{equation}
If (\ref{modn}) holds, then from the relation $f^{\sigma^i}(t)=f^{\sigma^{i+\p}}(t)$, we have that
\[
f(t)+f^{\sigma^{k}}(t)+\cdots+f^{\sigma^{(\p-1)k}}(t)=f(t)+f^{\sigma^{\mathrm{gcd}(k,\p)}}(t)+\cdots +f^{\sigma^{(\p-1)\mathrm{gcd}(k,\p)}}(t).
\]
First, if there exists an integer $m\in \{1,\cdots,\p-1\}$ such that $[m \frac{k}{\gcd(k,\p)}]=[0]$ holds, then $[m \gcd(k,\p)]=[0]$. Actually, if we write $m\frac{k}{\gcd(k,\p)}=q\p$, where $q \in \mathbb{Z}$, then the following formula holds.
\[
\begin{split}
m \gcd(k,\p)&=q\p \frac{\gcd(k,\p)}{k} \gcd(k,\p)\\
&= \p \frac{\gcd(qk\gcd(k,\p),q\p\gcd(k,\p))}{k}\\
&= \p \frac{\gcd(qk\gcd(k,\p),mk)}{k}\\
&= \p \gcd(q \gcd(k,\p),m).
\end{split}
\]
In other words, if $[m \frac{k}{\gcd(k,\p)}]=[0]$, then 
\[
[mk]=[m \frac{k}{\gcd(k,\p)} \gcd(k,\p)]=[0] \in \{[\mathrm{gcd}(k,\p)],[2\mathrm{gcd}(k,\p)],\cdots,[(\p-1)\mathrm{gcd}(k,\p)]\}.
\]
Next, we suppose that an integer $m \in \{1,\cdots,\p-1\}$ satisfies $[m\frac{k}{\gcd(k,\p)}]\neq[0]$. Then, there exists $m_k \in \{1,\cdots,\p-1\}$ with $[m_k]=[m\frac{k}{\gcd(k,\p)}]$. Hence,
\[
[mk]=[m\frac{k}{\gcd(k,\p)}\mathrm{gcd}(k,\p)]=[m_k\mathrm{gcd}(k,\p)] \in \{[\mathrm{gcd}(k,\p)],[2\mathrm{gcd}(k,\p)],\cdots,[(\p-1)\mathrm{gcd}(k,\p)]\}.
\]
Thus, $\{[k],[2k],\cdots,[(\p-1)k] \} \subset \{[\mathrm{gcd}(k,\p)],[2\mathrm{gcd}(k,\p)],\cdots,[(\p-1)\mathrm{gcd}(k,\p)]\}$.
Because the map
\[
\mapel{\phi_{k_{\p}}}{\{[\mathrm{gcd}(k,\p)],[2\mathrm{gcd}(k,\p)],\cdots,[(n-1)\mathrm{gcd}(k,\p)]\}}{\{[k],[2k],\cdots,[(\p-1)k]\}}{[x]}{[k_\p x]}
\]
is bijective, we have that $\{[k],[2k],\cdots,[(\p-1)k] \} = \{[\mathrm{gcd}(k,\p)],[2\mathrm{gcd}(k,\p)],\cdots,[(\p-1)\mathrm{gcd}(k,\p)]\}$. 
\end{proof}

We now prepare a lemma on greatest common divisors that will be used later.
\begin{lemma}\label{invgcd}
Let $\p$ and $\dg$ be integers. Then, for any integers $\mug_0 \in \mathbb{Z}$,
\begin{equation}
\gcd(\dg+\mug_0 \rad(\p) \gcd(\dg,\p),\p)=\gcd(\dg,\p),
\end{equation}
where $\rad(\p):=\underset{p:prime,p | \p}{\prod}p$ is a radical of $\p$.
\end{lemma}
\begin{proof}
Note that $\gcd(\frac{\dg}{\gcd(\dg,n)}, \frac{n}{\gcd(\dg,n)})=\frac{\gcd(\dg,n)}{\gcd(\dg,n)}=1$. Hence, for any integer $\mug_0$, we have that $\gcd(\frac{\dg}{\gcd(\dg,n)}+\mug_0 \rad(\p), \frac{n}{\gcd(\dg,n)})=1$. Therefore,
\[
\begin{split}
&\gcd(\dg+\mug_0 \rad(\p) \gcd(\dg,\p),\p)\\
&= \gcd(\dg,n)\gcd(\frac{\dg}{\gcd(\dg,n)}+\mug_0 \rad(\p), \frac{n}{\gcd(\dg,n)})\\
&=\gcd(\dg,\p).
\end{split}
\]
\end{proof}
\begin{proposition}\label{constituent_inv}
Let $f(t)$ be a quasi-polynomial of period $\p$ with the gcd-property. Let $k$ be a positive integer. Let $\tilde{f}^{\gcd(k,\p)}_j(t)$ be the constituent of the quasi-polynomial $\tilde{f}^{\gcd(k,\p)}(t)$ for $t \equiv j \bmod \gcd(k,\p)$. If $\gcd(j,\p)=1$
, then 
\begin{equation}
\tilde{f}^{\gcd(k,\p)}_j(t)=\tilde{f}^{\gcd(k,\rad(\p))}_j(t).
\end{equation}
\end{proposition}
\begin{proof}
We will prove that
\begin{equation}\label{set_gcd}
\{\gcd(j+\mug \gcd(k,\p),\p)\}_{\mug=0}^{\p-1}=\{\gcd(j+\mug \gcd(k,\rad(\p)),\p)\}_{\mug=0}^{\p-1}.	
\end{equation}
If (\ref{set_gcd}) holds, then from the gcd-property of $f(t)$,
\[
\begin{split}
&f_j(t)+f_{j+\gcd(k,\p)}(t)+\cdots+f_{j+(\p-1) \gcd(k,\p)}(t)\\
&= f_j(t)+f_{j+\gcd(k,\rad(\p))}(t)+\cdots+f_{j+(\p-1) \gcd(k,\rad(\p))}(t).
\end{split}
\]
Let $[b]:=b+\mathbb{Z}/n\mathbb{Z}$. Let $c:=\frac{\gcd(k,\p)}{\gcd(k,\rad(\p))}$. For any integer $\mug \in \{0,1,\cdots,\p-1\}$, there exists $\mug' \in \{0,1,\cdots,\p-1\}$ such that $[\mug']=[\mug c]$. Hence, $\{\gcd(j+\mug \gcd(k,\p),\p)\}_{\mug=0}^{\p-1} \subset \{\gcd(j+\mug \gcd(k,\rad(\p)),\p)\}_{\mug=0}^{\p-1}$. Next, we set $d:=\gcd(k,\p)$. We write $\p=\rg_1^{\sg_1}\rg_2^{\sg_2}\cdots \rg_m^{\sg_m}$ and $\dg= \rg_1^{\qg_1 \ig_1}\rg_2^{\qg_2 \ig_2}\cdots \rg_m^{\qg_m \ig_m}$, where $\rg_1,\rg_2,\cdots, \rg_m$ are primes, $\sg_1,\cdots,\sg_m,\qg_1,\cdots,\qg_m$ are positive integers, and $\ig_1,\cdots,\ig_m \in \{0,1\}$, and then we define $\invgcd_\p:=\rg_1^{\sg_1(1-\ig_1)} \rg_2^{\sg_2(1-\ig_2)}\cdots \rg_m^{\sg_m(1-\ig_m)}$. Note that $\gcd(d, \invgcd_\p)=1$ and any divisor of $\p$ that is relatively prime to $\dg$ divides $\invgcd_\p$. 
We have $\gcd(\frac{d}{\gcd(k,\rad(\p))},\frac{\rad(\p) \invgcd_\p}{\gcd(k,\rad(\p))})=1$
since $\gcd(d,\frac{\rad(\p)}{\gcd(k,\rad(\p))})=1$ and $\gcd(d,\invgcd_\p)=1$. Hence, for any integer $\mug \in \{0,1,\cdots,\p-1\}$, there exist integers $\mug_1,\mug_2 \in \mathbb{Z}$ such that $\mug=\mug_1\frac{d}{\gcd(k,\rad(\p))}+\mug_2\frac{\rad(\p) \invgcd_\p}{\gcd(k,\rad(\p))}$. We transform the formula
\begin{eqnarray}\label{aaa}
j+\mug \gcd(k,\rad(\p))&=&j+\Bigl(\mug_1\frac{d}{\gcd(k,\rad(\p))}+\mug_2\frac{\rad(\p) \invgcd_\p}{\gcd(k,\rad(\p))} \Bigr) \gcd(k,\rad(\p)) \nonumber \\
&=&j+\mug_1 d+\mug_2 \rad(\p) \invgcd_\p.
\end{eqnarray}
The integer $\gcd(j+\mug_1 d,\p)$ is relatively prime to $d$ since $\gcd(j+\mug_1 d,d)=1$. Since any divisor of $\p$ that is relatively prime to $\dg$ divides $\invgcd_\p$, $\gcd(j+\mug_1 \dg,\p)$ divides  $\invgcd_\p$. Let $\mug_3:=\frac{\invgcd_\p}{\gcd(j+\mug_1 d,\p)} \in \mathbb{Z}$. From (\ref{aaa}), we obtain
\begin{equation}\label{bbb}
j+\mug \gcd(k,\rad(\p))=j+\mug_1 d+\mug_2 \mug_3\rad(\p)\gcd(j+\mug_1 d,\p).
\end{equation}
Hence, using Lemma \ref{invgcd} for the right-hand side of (\ref{bbb}), we have the formula $\gcd(j+\mug \gcd(k,\rad(\p)),\p)=\gcd(j+\mug_1 \dg,\p)=\gcd(j+\mug_1 \gcd(k,\p),\p)$. Furthermore, since there exists an integer $\mug'_1 \in \{0,1,\cdots,\p-1\}$ such that $[\mug_1]=[\mug'_1]$, we have that $\{\gcd(j+\mug \gcd(k,\rad(\p)),\p)\}_{\mug=0}^{\p-1} \subset \{\gcd(j+\mug \gcd(k,\p),\p)\}_{\mug=0}^{\p-1}$.
\end{proof}

\begin{definition}\label{shift_bar}
Let $f(t)$ be a quasi-polynomial with period $\p$ as follows.\\
\[
\begin{split}
f(t) &= \left\{
\begin{array}{ll}
f_1(t), & t\equiv1 \bmod \p,\\
f_2(t), & t\equiv2 \bmod \p,\\
\quad \vdots\\
f_{\p-1}(t), & t\equiv \p-1 \bmod \p,\\
f_{\p}(t), & t\equiv 0 \bmod \p.
\end{array}\right.
\end{split}
\]
We define the operator $\overline{\Sh}$ as follows.
\[
\begin{split}
(\overline{\Sh}f) := \left\{
\begin{array}{ll}
f_1(t-1), & t\equiv1 \bmod \p,\\
f_2(t-1), & t\equiv2 \bmod \p,\\
\quad \vdots\\
f_{\p-1}(t-1), & t\equiv \p-1 \bmod \p,\\
f_{\p}(t-1), & t\equiv 0 \bmod \p.
\end{array}\right.
\end{split}
\]
\end{definition}
\begin{remark}
The operators $\Sh$ and $\overline{\Sh}$ have the relation
\[
\begin{split}
(\Sh f)(t) &= \left\{
\begin{array}{ll}
f_{\p}(t-1), & t\equiv1 \bmod \p,\\
f_1(t-1), & t\equiv2 \bmod \p,\\
\quad \vdots\\
f_{\p-2}(t-1), & t\equiv \p-1 \bmod \p,\\
f_{\p-1}(t-1), & t\equiv 0 \bmod \p,
\end{array}\right.\\
\\
&= \left\{
\begin{array}{ll}
f_{\sigma^{-1}(1)}(t-1),&t\equiv1 \bmod \p,\\
f_ {\sigma^{-1}(2)}(t-1),&t\equiv2 \bmod \p,\\
\quad \vdots\\
f_ {\sigma^{-1}(\p-1)}(t-1),&t\equiv \p-1 \bmod \p,\\
f_ {\sigma^{-1}(\p)}(t-1),&t\equiv 0 \bmod \p,
\end{array}\right.\\
\\
&= (\overline{\Sh}f^{\sigma})(t).
\end{split}
\]		
\end{remark}
\begin{lemma}\label{S_bar_linear}
\begin{enumerate}[(1)]
\item Let $f(t)$ and $g(t)$ be quasi-polynomials, which may have different minimal periods. Then, $\overline{\Sh}(f(t)+g(t))=\overline{\Sh}f(t)+\overline{\Sh}g(t)$, that is, the operator $\overline{\Sh}$ is linear. 
\item For any quasi-polynomial $h(t)$, $(\overline{\Sh}-1)^{\deg h+1}h(t)=0$.
\end{enumerate}
\end{lemma}
\begin{proof}
\begin{enumerate}[(1)]
\item Let $m$ and $n$ be the minimal periods of $f(t)$ and $g(t)$, respectively. Let $k$ be an integer. Let $f_k(t)$ and $g_k(t)$ be constituents of $f(t)$ and $g(t)$ for $t \equiv k \bmod \mathrm{lcm}(m,\p)$. If $t \equiv k \bmod \mathrm{lcm}(m,\p)$, then $\overline{\Sh}(f(t)+g(t))=f_k(t-1)+g_k(t-1)= \overline{\Sh}f(t)+\overline{\Sh}g(t)$. 
(2) By the definition of the operator $\overline{\Sh}$, the inequality $\deg ((\overline{\Sh}-1) h)<\deg h$ holds. Hence, inductively, $(\overline{\Sh}-1)^{\deg h+1} h(t)=0$.
\end{enumerate}
\end{proof}
\begin{lemma}\label{lemma_lemma}
Let $f(t)$ be a quasi-polynomial with period $\p$. Let $j$ and $m$ be integers. Let $\coxc$ be a multiple of $\frac{n}{\gcd(m,n)}$. Let $\sum_{k \equiv j \bmod c}\coef_k t^{mk}$ be a polynomial. Then,
\begin{equation}
(\sum_{k \equiv j \bmod \coxc}\coef_k\Sh^{mk}f)(t)=(\sum_{k \equiv j \bmod \coxc}\coef_k\overline{\Sh} ^{mk} f^{\sigma^{m j}})(t).
\end{equation}		
\end{lemma}
\begin{proof}
First, note that $(\Sh^{mj+m\coxc} f)(t)=(\overline{\Sh}^{mj+m\coxc} f^{\sigma^{mj +m\coxc}})(t) =(\overline{\Sh}^{mj+m\coxc}f^{\sigma^{mj}})(t)$ because $m\coxc$ is a multiple of $\p$.
\[
\begin{split}
(\sum_{k \equiv j \bmod \coxc}\coef_k\Sh^{mk}f)(t)&=(\coef_j\overline{\Sh}^{mj} f^{\sigma^{m j}})(t)+(\coef_{j+\coxc}\overline{\Sh}^{mj+m \coxc}f^{\sigma^{m j}})(t)+\cdots \\
&= (\sum_{k \equiv j \bmod \coxc}\coef_k\overline{\Sh}^{mk} f^{\sigma^{m j}})(t).
\end{split}
\]
\end{proof}
The following proposition concerns an average of  a quasi-polynomial using the cyclotomic shift operator.
\begin{proposition}\label{averaging}
Let $f(t)$ be a quasi-polynomial of degree $\degree$ with period $\p$. Let $g(t)$ be a polynomial. Let $m$ be an integer and $\coxc$ be a multiple of $\frac{\p}{\gcd(m,\p)}$. Then,
\begin{equation}
[\coxc]_{\Sh^m}^{\degree+1}g(\Sh^m)f(t)=[\coxc]_{\overline{\Sh}^m}^{\degree+1}g(\overline{\Sh}^m)\tilde{f}^{\mathrm{gcd}(m,\p)}(t).
\end{equation}
\end{proposition}
\begin{proof}
Let $[\coxc]_{{\Sh}^m}^{\degree+1}g({\Sh}^m)=:\sum_{k}\coef_k\Sh^{mk}$. We calculate $[\coxc]_{{\Sh}^m}^{\degree+1}g({\Sh}^m)f(t)$ using Lemma \ref{Athanasiadis's lemma_2}, Proposition \ref{tilde-gcd}, Lemma \ref{S_bar_linear}, and Lemma \ref{lemma_lemma}.\\
\[
\begin{split}
([\coxc]_{\Sh^{m}}^{\degree+1}g(\Sh^m)f)(t)&=(\sum_{k}\coef_{k}\Sh^{mk}f)(t)\\
&=(\sum_{k \equiv 0 \bmod \coxc}\coef_k \Sh ^{mk} f)(t) +\cdots+ (\sum_{k \equiv \coxc-1 \bmod \coxc}\coef_k\Sh^{mk}f)(t)\\
&=(\sum_{k \equiv 0 \bmod \coxc}\coef_k \overline{\Sh}^{mk} f)(t) +\cdots+ (\sum_{k \equiv \coxc-1 \bmod \coxc}\coef_k \overline{\Sh}^{mk} f^{\sigma^{m(\coxc-1)}})(t)\\
&=(\frac{1}{\coxc}\sum_{k}\coef_k\overline{\Sh}^{mk} f)(t)+\cdots+(\frac{1}{\coxc}\sum_{k}\coef_k\overline{\Sh}^{mk} f^{\sigma^{m (\coxc-1)}})(t)\\
&=(\frac{1}{\coxc} \sum_{k}\coef_k\overline{\Sh}^{mk})(f(t)+ f^{\sigma^{m}}(t)+\cdots+f^{\sigma^{m (\coxc-1)}}(t))\\
&=(\sum_{k}\coef_k\overline{\Sh}^{mk}) \Biggl (\frac{f(t)+ f^{\sigma^{ m}}(t)+\cdots+f^{\sigma^{m (\coxc-1)}}(t)}{\coxc} \Biggr)\\
&=([c]_{\overline{\Sh}^m}^{\degree+1}g(\overline{\Sh}^m)\tilde{f}^{m})(t)\\
&=([c]_{\overline{\Sh}^m}^{\degree+1}g(\overline{\Sh}^m)\tilde{f}^{\gcd(m,\p)})(t).
\end{split}
\]
\end{proof}
\subsection{Decomposition of a quasi-polynomial}\label{sec:Pre_deco}
First, we summarize the relation between (quasi-)polynomial and rational functions.
\begin{lemma}\label{gene_poly}(\cite[Corollary 4.3.1]{Beck-Robinson, Stanley-EC1})
If
\[
\sum_{n=0}^{\infty}f(n)\x^n=\frac{g(\x)}{(1-\x)^{\degree+1}},
\]
then $f(t)$ is a polynomial of degree $\degree$ if and only if $g(\x)$ is a polynomial of degree at most $\degree$ and cannot be divided by $(1-\x)$.
\end{lemma}
\begin{lemma}\label{gene_quasi_p}(\cite[Proposition 4.4.1]{Beck-Robinson, Stanley-EC1})
If 
\[
\sum_{n=0}^{\infty}f(n)\x^n=\frac{g(\x)}{h(\x)},
\]
then $f(t)$ is a quasi-polynomial of degree $\degree$ with period $\pe$ if and only if $g(\x)$ and $h(\x)$ are polynomials such that $\deg g<\deg h$ and all roots of $h(x)$ are $\pe$-th roots of unity of multiplicity at most $\degree+1$, and there is a root of multiplicity equal to $\degree+1$ (all of this assuming that $\frac{g(\x)}{h(x)}$ has been reduced to its lowest terms).
\end{lemma}
The following classical lemma is called partial fraction decomposition.
\begin{lemma}\label{pfd}
Let $g(\x)$ and $h(\x)$ be polynomials with $\deg g < \deg h$. Let $h_1(\x),\cdots,h_n(\x)$ be polynomials with $h(\x)=h_1(\x)h_2(\x)\cdots h_n(\x)$ that are relatively prime to each other. Then, there exist polynomials $g_1(\x),\cdots,g_n(\x)$ such that $\deg g_i < \deg h_i$ for any $i\in\{1,\cdots,n\}$ and  
\begin{equation}
\frac{g(\x)}{h(\x)}=\frac{g_1(\x)}{h_1(\x)}+\cdots + \frac{g_n(\x)}{h_n(\x)}.
\end{equation}
\end{lemma}
In general, a quasi-polynomial has the following decomposition into several quasi-polynomials.
\begin{proposition}\label{gene_quasi_deco}
Let $g(\x)$ and $h(\x)$ be relatively prime polynomials such that $\deg g<\deg h$ and all roots of $h(\x)$ are $\pe$-th roots of unity. Let $\divisors{\pe}$ be the set of divisors of $\pe$. Let $\mult_{i}+1$ be the number of primitive $i$-th roots of unity in the roots of $h(\x)$. If 
\[
\sum_{n=0}^{\infty}f(n)\x^n=\frac{g(\x)}{h(\x)},
\]
then there exist quasi-polynomials $f_{i}^{(\mult_i)}(t)$, $(i \in \divisors{\pe})$ of degree $\mult_i$ and period $i$ that satisfy
\begin{equation}\label{claim}
f(t)=\sum_{i \in \divisors{\pe}}f_{i}^{(\mult_i)}(t).
\end{equation}
\end{proposition}
\begin{proof}
Let $\{h_i(\x)\}_{i \in \divisors{\pe}}$ be polynomials such that all roots of $h_i(\x)$ are primitive $i$-th roots of unity in the roots of $h(\x)$ and $h(\x)=\prod_{i \in \divisors{\pe}}h_i(\x)$. By Lemma \ref{pfd}, there exist polynomials $\{g_i(\x)\}_{i\in \divisors{\pe}}$ such that $\deg g_i< \deg h_i=\mult_{i}+1$ for any $i \in \divisors{\pe}$ and 
\[
\sum_{n=0}^{\infty}f(n)\x^n=\sum_{i\in \divisors{\pe}}\frac{g_i(\x)}{h_{i}(\x)}.
\]
By Lemma \ref{gene_quasi_p}, for any $i\in \divisors{\pe}$, there exists the quasi-polynomial $f^{(\mult_i)}_{i}(t)$ of degree $\mult_i$ with period $i$ such that $\sum_{n=0}^{\infty}f^{(\mult_i)}_{i}(n)\x^n=\frac{g_i(\x)}{h_{i}(\x)}$. Hence, 
\begin{equation}\label{above}
\sum_{n=0}^{\infty}f(n)\x^n=\sum_{i \in \divisors{\pe}}\sum_{n=0}^{\infty}f^{(\mult_i)}_{i}(n)\x^n.
\end{equation}
By comparing each term of (\ref{above}), we obtain the formula stated in (\ref{claim}).
\end{proof}
\subsection{Root system}\label{root system}
We now introduce some concepts that help to explain the results for the characteristic polynomial of the Linial arrangement given by Yoshinaga \cite{Yoshinaga_1}. Let $V=\mathbb{R}^{\degree}$ be the Euclidean space with inner product $(\cdot,\cdot)$. Let $\Phi \subset V$ be an irreducible root system with Coxeter number $\coxn$. Fix a positive system $\Phi^{+}\subset \Phi$ and the set of simple roots $\Delta=\{\alpha_1,\cdots,\alpha_{\degree}\} \subset \Phi^{+}$. The highest root, denoted by $\tilde{\alpha} \in \Phi^{+}$, can be expressed as the linear combination $\tilde{\alpha}=\sum_{i=1}^{\degree}\coxc_i \alpha_i$ $(\coxc_i \in \mathbb{Z}_{>0})$. We also set $\alpha_0:=-\tilde{\alpha}$ and $\coxc_0:=1$. Then, we have the linear relation
\begin{equation}
\coxc_0\alpha_0+ \coxc_1\alpha_1+\cdots+ \coxc_{\degree}\alpha_{\degree}=0.	
\end{equation}
The integers $\coxc_0,\cdots, \coxc_{\degree}$ have the following relation with the Coxeter number $h$:
\begin{proposition}\label{coxc_coxn_relation}(\cite{Humphreys})
\begin{equation}
\coxc_0+\coxc_1+\cdots+\coxc_{\degree}=\coxn.
\end{equation}
\end{proposition}
\subsection{Ehrhart quasi-polynomial for the fundamental alcove}\label{section:Eh_quasi}
The coweight lattice $Z(\Phi)$ and the coroot lattice $Q(\Phi)$ are defined as follows.
\[
\begin{split}
Z(\Phi)&:=\Set{x\in V}{(\alpha_i,x) \in \mathbb{Z}, \alpha_i\in \Delta},\\
Q(\Phi)&:=\sum_{\alpha \in \Phi}\mathbb{Z}\cdot\frac{2\alpha}{(\alpha,\alpha)}.
\end{split}
\]
The index $\#\frac{Z(\Phi)}{Q(\Phi)}=f$ is called the index of connection. Let $\varpi_i \in Z(\Phi)$ be the dual basis for the simple roots $\alpha_1,\cdots,\alpha_{\degree}$, that is, $(\alpha_i,\varpi_j)=\delta_{ij}$. Then, $Z(\Phi)$ is a free abelian group generated by $\varpi_1,\cdots,\varpi_{\degree}$. We also have $c_i=(\varpi_i,\tilde{\alpha})$.
A connected component of $V\verb|\| \underset{{\underset{k\in\mathbb{Z}}{\alpha\in \Phi^{+}}}}{\cup}H_{\alpha,k}$ is called an alcove. Let us define the fundamental alcove $F_{\Phi}$ of type $\Phi$ as
\[
\begin{split}
F_{\Phi}:=\left\{
 x\in V\ \middle|
\begin{array}{ll}
&(\alpha_i,x)>0,\ (1\leqq i\leqq\degree)\\
&(\tilde{\alpha},x)<1
\end{array}
\right\}.
\end{split}
\]
The closure $\overline{F_{\Phi}}=\Set{x \in V}{(\alpha_i,x)\geqq0\ (1\leqq i\leqq\degree),\ 
(\tilde{\alpha},x)\leqq 1
}$ is the convex hull of $0,\frac{\varpi_1}{c_1},\cdots, \frac{\varpi_{\degree}}{c_{\degree}}\in V$. The closed fundamental alcove $\overline{F_{\Phi}}$ is a simplex. For a positive integer $q \in \mathbb{Z}_{>0}$, we define the function $\map{\Eh_{\Phi}}{\mathbb{Z}_{>0}}{\mathbb{Z}}$ as
\begin{equation}
\Eh_{\Phi}(q):=\#(q F_{\Phi}\cap Z(\Phi)).
\end{equation}
The function $\Eh_{\Phi}(q)$ can be extended as the function $\map{\Eh_{\Phi}}{\mathbb{Z}}{\mathbb{Z}}$
because $\Eh_{\Phi}(q)$ is a quasi-polynomial \cite{Beck-Robinson}. The quasi-polynomial $\Eh_{\Phi}(t)$ is called the Ehrhart quasi-polynomial for the fundamental alcove of type $\Phi$. Let $\Ehp$ be the minimal period of the quasi-polynomial $\Eh_{\Phi}(t)$. The quasi-polynomial $\Eh_{\Phi}(t)$ was computed for any irreducible root system $\Phi$ by Suter \cite{Suter}. In particular, for type $A_{\degree}$, the Ehrhart quasi-polynomial $\Eh_{A_{\degree}}(t)=\binom{t+\degree}{\degree}=\frac{(t+1)\cdots(t+\degree)}{\degree !}$. The Ehrhart quasi-polynomial $\Eh_{\Phi}(t)$ satisfies the following duality. 
\begin{theorem}(Suter \cite{Suter})
Let $\Phi$ be an irreducible root system of rank $\degree$.
If $q \in \mathbb{Z}$, then 
\begin{equation}
\Eh_{\Phi}(-q)=(-1)^{\degree}\Eh_{\Phi}(q-h).
\end{equation}
\end{theorem}
The following statements are true for the Ehrhart quasi-polynomial $\Eh_{\Phi}(t)$.
\begin{theorem}\label{gene_Eh}(Suter \cite{Suter})
\begin{enumerate}[(1)]
\item 	The Ehrhart quasi-polynomial $\Eh_{\Phi}(t)$ has the gcd-property.
\item The degree of $\Eh_{\Phi}(t)$ is the rank of $\Phi$
\item The minimal period $\Ehp$ is as given in Table \ref{fig:table_Ehpara}.
\item $\Eh_{\Phi}(-1)=\Eh_{\Phi}(-2)=\cdots= \Eh_{\Phi}(-(\coxn-1))=0$.
\item The generating function of $\Eh_{\Phi}(t)$ is 
\begin{equation}
\sum_{n=0}^{\infty}\Eh_{\Phi}(n)\x^{n}=\frac{1}{(1-\x^{\coxc_0})\cdots (1-\x^{\coxc_\degree})}.
\end{equation}
\end{enumerate}
\end{theorem}
There is a relation between the Ehrhart quasi-polynomials of each type.
\begin{proposition}\label{root Ehrhart}
Let $\Phi$ be an irreducible root system of rank $\degree$.
The following formula holds.
\begin{equation}\label{cyclo_Ehrhart}
\Eh_{A_{\degree}}(t)=\C{\Sh}\Eh_{\Phi}(t).
\end{equation}
\end{proposition}
\begin{proof}
The Ehrhart polynomial for the fundamental alcove of any irreducible root system $\Phi$ satisfies $\Eh_{\Phi}(-1)= \Eh_{\Phi}(-2)=\cdots=\Eh_{\Phi}(-(h-1))=0$. Hence, by Proposition \ref{coxc_coxn_relation},
\begin{equation}\label{Eh_coxc_shift}
[\coxc_0]_\x \cdots[\coxc_{\degree}]_\x \sum_{n=0}^{\infty}\Eh_{\Phi}(n)\x^n=\sum_{n=0}^{\infty}([\coxc_0]_{\Sh}\cdots[\coxc_{\degree}]_{\Sh}\Eh_{\Phi})(n)\x^n.
\end{equation}
On the left-hand side of (\ref{Eh_coxc_shift}), by Theorem \ref{gene_Eh}, we can write  
\[
\begin{split}
[\coxc_0]_\x \cdots[\coxc_{\degree}]_\x \sum_{n=0}^{\infty}\Eh_{\Phi}(n)\x^n&=[\coxc_0]_\x \cdots[\coxc_{\degree}]_\x \frac{1}{(1-\x^{\coxc_0})\cdots(1-\x^{\coxc_{\degree}})}\\
&=\frac{1}{(1-\x)^{\degree+1}}\\
&=\sum_{n=0}^{\infty}\Eh_{A_{\degree}}(n)\x^n.
\end{split}
\]
Therefore, by comparing each term of (\ref{Eh_coxc_shift}), we obtain the formula given in (\ref{cyclo_Ehrhart}).
\end{proof}
\begin{remark}
Note that $(1-\Sh)\Eh_{A_{\degree}}(t)= \Eh_{A_{\degree-1}}(t)$. We obtain the relations between the Ehrhart quasi-polynomials of root systems of different ranks from (\ref{cyclo_Ehrhart}) and Proposition \ref{Shift congruences}. The following are some examples.
\[
(1-\Sh^2)\Eh_{C_{\degree}}(t)=\Eh_{C_{\degree-1}}(t).
\]
\[
(1-\Sh^2)\Eh_{D_{\degree}}(t) =\Eh_{D_{\degree-1}}(t).
\]
\[
[3]_{\Sh}[4]_{\Sh}(1-\Sh)\Eh_{E_{7}}(t) =\Eh_{E_{6}}(t).
\]
\[
[2]_{\Sh^2}[5]_{\Sh}[6]_{\Sh}(1-\Sh)\Eh_{E_{8}}(t)=\Eh_{E_{7}}(t).
\]
\[
[2]_{\Sh}[4]_{\Sh}(1-\Sh)^2\Eh_{F_{4}}(t) =\Eh_{G_{2}}(t).
\]
\[
(1-\Sh)^2\Eh_{E_{6}}(t) =(1+\Sh^2)\Eh_{F_{4}}(t).
\]
\end{remark}
Let $\difi_0,\cdots,\difi_{\ndifi}$ be all the different integers in $\coxc_0,\cdots, \coxc_{\degree}$ and $\adegree_{\difi_k}+1$ be the number of multiples of $\difi_k$ in $\coxc_0,\cdots, \coxc_{\degree}$. Theorem \ref{gene_Eh} and Proposition \ref{gene_quasi_deco} lead to the following decomposition of the Ehrhart quasi-polynomial $\Eh_{\Phi}(t)$.
\begin{proposition}\label{Eh_deco}
For any irreducible root system $\Phi$, there exist quasi-polynomials $\{\Ehf{k}{(\adegree_k)}\}_{k \in \{\difi_0,\cdots, \difi_{\ndifi}\}}$ such that
\begin{equation}\label{Eh_sum}
\Eh_{\Phi}(t)=\sum_{k\in \{\difi_0,\cdots, \difi_{\ndifi}\}}\Ehf{k}{(\degree_{k})}(t),
\end{equation}	
where $\Ehf{k}{(\adegree_{k})}$ has period $k$ and degree $\adegree_{k}$.
\end{proposition}
\begin{remark}
We can express Proposition \ref{Eh_deco} in a different way from formula (\ref{Eh_sum}). First, note that a quasi-polynomial $f(t)$ of degree $\degree$ can be expressed with the periodic functions $\pf^{0}(t),\cdots,\pf^{\degree}(t)$ as follows.
\[
f(t)=\pf^{\degree}(t)t^{\degree}+\cdots+\pf^1(t)t+\pf^0(t).
\]
Let $\pf_i^{j}(t)$ be a periodic function with period $i$.
In the case of $E_6$, Proposition \ref{Eh_deco} can be expressed as follows.
\[
\begin{split}
\Ehf{1}{(6)}(t)&=\pf_1^6(t)t^6+ \pf_1^5(t)t^5+ \pf_1^4(t)t^4+ \pf_1^3(t)t^3+ \pf_1^2(t)t^2+ \pf_1^1(t)t+ \pf_1^0(t).\\
\Ehf{2}{(2)}(t)&=\pf_2^2(t)t^2+ \pf_2^1(t)t+ \pf_2^{0}(t).\\
\Ehf{3}{(0)}(t)&=\pf_3^{0}(t).
\end{split}
\]
\[
\begin{split}	
\Eh_{E_6}(t)&=\Ehf{1}{(6)}(t)+\Ehf{2}{(2)}(t)+\Ehf{3}{(0)}(t)\\
&=\pf_{1}^{6}(t)t^6+ \pf_1^{5}(t)t^5+\pf_1^{4}(t)t^4+\pf_1^{3}(t)t^3+\Bigl(\pf_1^{2}(t)+\pf_2^{2}(t)\Bigr)t^2\\
&\quad +\Bigl(\pf_1^{1}(t)+\pf_2^{1}(t)\Bigr)t+\Bigl(\pf_1^{0}(t)+\pf_2^{0}(t)+\pf_3^{0}(t)\Bigr).
\end{split}
\]
From this expression, we can see that the parts of degree $6,5,4$, and $3$ have the period $1$, the parts of degree $2$ and $1$ have the period $2$, and the parts of degree $0$ have the period $6$ since the period of the sum of periodic functions is the least common multiple of the period of each periodic function. Note that $\{\Ehf{k}{(\adegree_k)}\}_{k \in \{\difi_0,\cdots,\difi_{\ndifi}\}}$ are not unique, because it is sufficient for part of a periodic function of the quasi-polynomial $\Eh_{\Phi}(t)$ to be the sum of periodic functions.
\end{remark}

\begin{table}[htbp]
\centering
\caption{Table of root systems.}
{\footnotesize 
\begin{tabular}{c|l|l|l|c|c|c|c}
$\Phi$&$\coxc_0, \cdots, \coxc_\degree$&$\difi_0, \cdots, \difi_{\ndifi}$&$\adegree_{\difi_0}, \cdots, \adegree_{\difi_{\ndifi}} $&$\ndifi$&$\Ehp$&$\rad(\Ehp)$&$\coxn$\\
\hline\hline
$A_\degree$&$1,1,1,\dots,1$&$1$&$\degree$&$1$&$1$&$1$&$\degree+1$\\
$B_\degree, C_\degree$&$1,1,2,2,\dots,2$&$1,2$&$\degree,\degree-2$&$2$&$2$&$2$&$2\degree$\\
$D_\degree$&$1,1,1,1,2,\dots,2$&$1,2$&$\degree,\degree-4$&$2$&$2$&$2$&$2\degree-2$\\
$E_6$&$1,1,1,2,2,2,3$&$1,2,3$&$6,2,0$&$3$&$6$&$6$&$12$\\
$E_7$&$1,1,2,2,2,3,3,4$&$1,2,3,4$&$7,3,1,0$&$4$&$12$&$6$&$18$\\
$E_8$&$1,2,2,3,3,4,4,5,6$&$1,2,3,4,5,6$&$8,4,2,1,0,0$&$6$&$60$&$30$&$30$\\
$F_4$&$1,2,2,3,4$&$1,2,3,4$&$4,2,0,0$&$4$&$12$&$6$&$12$\\
$G_2$&$1,2,3$&$1,2,3$&$2,0,0$&$3$&$6$&$6$&$6$
\end{tabular}
}
\label{fig:table_Ehpara}
\end{table}
\newpage

\subsection{Eulerian polynomial}\label{section:generalized Eulerian}
We summarize some facts about the Eulerian polynomial and the generalized Eulerian polynomial with reference to \cite{Yoshinaga_1}.
\begin{definition}[Eulerian polynomial]
For a permutation $\tau \in \mathfrak{S}_{n} $, define
\[
\ascl(\tau):=\#\Set{i\in \{1,\cdots,n-1\}}{\tau(i)<\tau(i+1)}.
\]
Then,
\begin{equation}
\A(n,k):=\#\Set{\tau \in \mathfrak{S}_{n}}{\ascl(\tau)=k-1}
\end{equation}
$(1\leqq k\leqq n)$ is called the Eulerian number and the generating polynomial
\begin{equation}
\A_{n}(t):=\sum_{k=1}^{n}\A(n,k)t^k=\sum_{\tau \in \mathfrak{S}_{n}}t^{1+\ascl(\tau)}
\end{equation}
is called the Eulerian polynomial. Define $\A_0(t)=1$.
\end{definition}
The Eulerian polynomial $\A_{\degree}(t)$ satisfies the duality $\A_{\degree}(t)=t^{\degree+1}\A_{\degree}(\frac{1}{t})$. The following theorem is the so-called Worpitzky identity.
\begin{theorem}(Worpitzky \cite{Worpitzky})
Note that $\Eh_{A_{\degree}}(t)=\binom{t+\degree}{\degree}=\frac{(t+1)\cdots(t+\degree)}{\degree!}$. Then,
\begin{equation}\label{Worpitzky identity}
t^{\degree}=\A_{\degree}(\Sh)\Eh_{A_{\degree}}(t).
\end{equation}
\end{theorem}
The Eulerian polynomial also satisfies the following congruence.
\begin{theorem}\label{Eulerian}(\cite{Iijima-Sasaki-Takahashi-Yoshinaga}, \cite{Yoshinaga_1})
Let $\degree \geqq 1$, $n \geqq 2$. Then,
\begin{equation}
\A_{\degree}(t^n) \equiv \frac{1}{n^{\degree+1}}[n]_{t}^{\degree+1}\A_{\degree}(t)\  \bmod \ (1-t)^{\degree+1}.
\end{equation}
\end{theorem}
Lam and Postnikov introduced the following generalization of Eulerian polynomials \cite{Lam-Postnikov}.
\begin{definition}[Generalized Eulerian polynomial]
Let $W$ be the Weyl group of an irreducible root system $\Phi$. For $\omega\in W$, the integer $\asc(\omega) \in \mathbb{Z}$ is defined by 
\[
\asc(\omega):=\sum_{\underset{\omega(\alpha_i)>0}{0\leqq i\leqq \degree}}c_i.
\]
Then,
\begin{equation}
\R_{\Phi}(t):=\frac{1}{f}\sum_{\omega \in W}t^{\asc(\omega)}	
\end{equation}
is called the generalized Eulerian polynomial of type $\Phi$.
\end{definition}
The generalized Eulerian polynomial $\R_{\Phi}(t)$ can be expressed in terms of the cyclotomic type polynomial $[c]_t$ and the Eulerian polynomial $\A_{\degree}(t)$.
\begin{theorem}\label{Lam-Postnikov}(Lam--Postnikov \cite{Lam-Postnikov}, Theorem 10.1)
Let $\Phi$ be an irreducible root system of rank $\ell$. Then, 
\begin{equation}
\R_{\Phi}(t)=[\coxc_0]_{t}[\coxc_1]_{t} \cdots [\coxc_{\degree}]_{t}\A_ {\degree}(t).
\end{equation}
\end{theorem}
Some basic properties of the generalized Eulerian polynomial $\R_{\Phi}(t)$ follow from Theorem \ref{Lam-Postnikov} (Lam--Postnikov \cite{Lam-Postnikov}).
\begin{proposition}
\begin{enumerate}[(1)]
\item $\deg \R_{\Phi}=\coxn-1$.
\item $t^{\coxn}\R_{\Phi}(\frac{1}{t})=\R_{\Phi}(t)$.
\item $\R_{A_{\degree}}(t)=\A_{\degree}(t)$.
\end{enumerate}
\end{proposition}
We can obtain the following formula from Theorems \ref{Eulerian} and \ref{Lam-Postnikov}.
\begin{proposition}\label{g_Eulerian_cong}
Let $\Phi$ be an irreducible root system of rank $\degree$. Let $n$ be a positive integer. Then,
\begin{equation}
\R_{\Phi}(t^{n})\equiv (\prod_{i=0}^{\degree}\frac{1}{n}\cyc{n}{t^{\coxc_i}})\R_{\Phi}(t) \bmod (1-t)^{\degree+1}.
\end{equation}
\end{proposition}
\begin{proof}
Using Theorems \ref{Eulerian} and \ref{Lam-Postnikov}, we calculate the following.
\[
\begin{split}
\R_{\Phi}(t^{n})&=[\coxc_0]_{t^{n}}[\coxc_1]_{t^{n}} \cdots [\coxc_{\degree}]_{t^{n}}\A_{\degree}(t^{n})\\
&\equiv [\coxc_0]_{t^{n}}[\coxc_1]_{t^{n}} \cdots [\coxc_{\degree}]_{t^{n}}(\frac{1}{n^{\degree+1}}[n]_{t}^{\degree+1}\A_{\degree}(t)) \bmod (1-t)^{\degree+1}\\
&\equiv \frac{1}{n^{\degree+1}}[\coxc_0 n]_{t}[\coxc_1 n]_{t} \cdots [\coxc_{\degree} n]_{t}\A_{\degree}(t) \bmod (1-t)^{\degree+1}\\
&\equiv \frac{1}{n^{\degree+1}}[n]_{t^{\coxc_0}} [n]_{t^{\coxc_1} }\cdots [n]_{t^{\coxc_{\degree}}}[\coxc_0]_{t}[\coxc_1]_{t} \cdots [\coxc_{\degree}]_{t}\A_{\degree}(t) \bmod (1-t)^{\degree+1}\\
&\equiv \frac{1}{n^{\degree+1}}[n]_{t^{\coxc_0}} [n]_{t^{\coxc_1} }\cdots [n]_{t^{\coxc_{\degree}}} \R_{\Phi}(t) \bmod (1-t)^{\degree+1}.\\
\end{split}
\]
\end{proof}
\subsection{Postnikov--Stanley Linial arrangement conjecture}\label{section:Conjecture}
Let $V$ be a vector space with the inner product $(\cdot,\cdot)$. For any integer $k\in \mathbb{Z}$ and $\alpha \in V$, the affine hyperplane $H_{\alpha,k}$ is defined by
\begin{equation}
H_{\alpha,k}:=\Set{x \in V}{(\alpha,x)=k}.\end{equation}
Let $a,b\in \mathbb{Z}$ be integers with $a\leqq b$. Define the hyperplane arrangement $\mathcal{A}_{\Phi}^{[a,b]}$ as follows.
\begin{equation}
\mathcal{A}_{\Phi}^{[a,b]}:=\Set{H_{\alpha,k}}{\alpha \in \Phi^{+}, k \in \mathbb{Z}, a\leqq k\leqq b}. 
\end{equation}
Note that we define $\mathcal{A}_{\Phi}^{[1,0]}$ as an empty set. The hyperplane arrangement $\mathcal{A}_{\Phi}^{[a,b]}$ is called the truncated affine Weyl arrangement. In particular, $\mathcal{A}_{\Phi}^{[1,\nn]}$ is called the Linial arrangement. Let us denote by $\chi(\mathcal{A}_{\Phi}^{[a,b]},t)$ the characteristic polynomial of $\mathcal{A}_{\Phi}^{[a,b]}$. Postnikov and Stanley conjectured the following for $\chi(\mathcal{A}_{\Phi}^{[a,b]},t)$.
\begin{conjecture}\label{Postnikov-Stanley_2}(Postnikov--Stanley \cite{Postnikov-Stanley}, Conjecture 9.14)
Let $a,b \in \mathbb{Z}$ with $a \leqq 1 \leqq b$. Suppose that $1 \leqq a+b$. Then, every root $z \in \mathbb{C}$ of the equation $\chi(\mathcal{A}_{\Phi}^{[a,b]},t)=0$ satisfies $\Re z=\frac{(b-a+1)\coxn}{2}$.
\end{conjecture}
It is known that if Conjecture \ref{Postnikov-Stanley_2} is true in the case of the Linial arrangement $\mathcal{A}_{\Phi}^{[1,n]}$, then Conjecture \ref{Postnikov-Stanley_2} is also true by the following theorem.
\begin{theorem}(Yoshinaga \cite{Yoshinaga_1})
Let $n \geqq 0$ and $k\geqq 0$. The characteristic quasi-polynomial of the Linial arrangement $\mathcal{A}_{\Phi}^{[1,n]}$ is
\begin{equation}\label{Ch_para_shift_2}
\chi(\mathcal{A}_{\Phi}^{[1,n]},t)=\chi(\mathcal{A}_{\Phi}^{[1-k,n+k]},t+kh).
\end{equation} 	
\end{theorem}
For classical root systems, the formula in (\ref{Ch_para_shift_2}) has been proved by Athanasiadis \cite{Athanasiadis_0, Athanasiadis}.\par
Conjecture \ref{Postnikov-Stanley_2} was proved by Postnikov and Stanley for $\Phi=A_{\degree}$ \cite{Postnikov-Stanley}, and by Athanasiadis for $\Phi= A_{\degree},B_{\degree},C_{\degree},D_{\degree}$ \cite{Athanasiadis}. Yoshinaga \cite{Yoshinaga_1} verified Conjecture \ref{Postnikov-Stanley_2} for $E_6,E_7,E_8,F_4$ when the parameter $n>0$ of the Linial arrangement $\mathcal{A}_{\Phi}^{[1,\nn]}$ satisfies
\begin{equation}\label{the parameter}
n \equiv -1 	\left\{
\begin{array}{ll}
\bmod \quad 6, &\Phi=E_6, E_7, F_4\\
\bmod \quad 30, &\Phi=E_8.
\end{array}\right.
\end{equation}
He also verified Conjecture \ref{Postnikov-Stanley_2} for exceptional root systems when the parameter $\nn$ is a sufficiently large integer \cite{Yoshinaga_2}. The case $\Phi=G_2$ is easy.\par
In proving the conjecture for the case in (\ref{the parameter}), Yoshinaga studied from the perspective of the characteristic quasi-polynomial \cite{Yoshinaga_1}, which was introduced by Kamiya et al.~\cite{Kamiya-Takemura-Terao_0}. One of the most important properties of the characteristic quasi-polynomial is that it coincides with the characteristic polynomial on the integers that are relatively prime to its own period as a quasi-polynomial. Let us denote by $\chi_{quasi}(\mathcal{A}_{\Phi}^{[1,\nn]},t)$ the characteristic quasi-polynomial of $\mathcal{A}_{\Phi}^{[1,\nn]}$. Yoshinaga proved the explicit formula for $\chi_{quasi}(\mathcal{A}_{\Phi}^{[1,\nn]},t)$.
\begin{theorem}\label{characteristic_quasi_poly}(Yoshinaga \cite{Yoshinaga_1})
Let $\nn \geqq 0$. The characteristic quasi-polynomial of the Linial arrangement $\mathcal{A}_{\Phi}^{[1,\nn]}$ is
\begin{equation}\label{Ch_Yoshinaga}
\chi_{quasi}(\mathcal{A}_{\Phi}^{[1,\nn]},t)=\R_{\Phi}(\Sh^{\nn+1})\Eh_{\Phi}(t).
\end{equation}
\end{theorem}
From (\ref{Ch_Yoshinaga}), we see that $\chi_{quasi}(\mathcal{A}_{\Phi}^{[1,\nn]},t) $ has the same period as $\Eh_{\Phi}(t)$, namely, the period $\Ehp$. Note that $\chi_{quasi}(\mathcal{A}_{\Phi}^{[1,\nn]},t)= \chi(\mathcal{A}_{\Phi}^{[1,\nn]},t)$ when $t \equiv 1 \bmod \Ehp$. We will calculate the left-hand side of (\ref{Ch_Yoshinaga}) in the following section.\par
When $\nn=0$, that is, $\mathcal{A}_{\Phi}^{[1,0]}=\emptyset$, Theorem \ref{characteristic_quasi_poly} leads to the following generalization of the Worpitzky identity (\ref{Worpitzky identity}) \cite{Yoshinaga_1,Yoshinaga_2}.
\begin{theorem}\label{g-Eulerian}(Yoshinaga \cite{Yoshinaga_1})
\begin{equation}
t^{\degree}=\R_{\Phi}(\Sh)\Eh_{\Phi}(t).
\end{equation}
\end{theorem}
\section{Main results}
\subsection{Postnikov--Stanley Linial arrangement conjecture when the parameter $\nn+1$ is relatively prime to the period}\label{section:main_formula}
\begin{theorem}\label{main theorem}
Let $n\geqq0$.
\begin{equation}
\chi _{quasi}(\mathcal{A}_{\Phi}^{[1,\nn]},t)=\R_{\Phi}(\overline
{\Sh}^{\nn+1})\avEh{\Phi}{\mathrm{gcd}(\nn+1,\Ehp)}(t).
\end{equation}
In particular, $\chi _{quasi}(\mathcal{A}_{\Phi}^{[1,\nn]},t)$ has the period $\gcd(\nn+1,\Ehp)$.
\end{theorem}
\begin{proof}
Let $\Phi$ be an irreducible root system of rank $\degree$. We can define the polynomial
\[
g_k(t^{\nn+1}):= \frac{\C{t^{\nn+1}} \A_{\degree}(t^{\nn+1})}{[\difi_k]_{t^{\nn+1}}^{\adegree_{\difi_k}+1}}
\]
for any $k\in\{0,\cdots,\ndifi\}$ because $[\difi_k]_{t^{\nn+1}}^{\adegree_{\difi_k}+1}$ divides $\C{t^{\nn+1}}$. By Proposition \ref{Eh_deco},
\begin{equation}
\Eh_{\Phi}(t)=\sum_{k\in \{\difi_0,\cdots, \difi_{\ndifi}\}}\Ehf{k}{(\adegree_{k})}(t).
\end{equation}	
Note that $\Ehf{\difi_k}{(\adegree_{\difi_k})}(t)$ is a quasi-polynomial of degree $\adegree_{\difi_k}$ with period $\difi_k$. Because $\Ehp$ is a multiple of $\difi_k$, by Proposition \ref{tilde-gcd}, we obtain $\widetilde{\Ehf{\difi_k} {(\adegree_{\difi_k})}}^{\gcd(\nn+1,\Ehp)} (t)=\widetilde{\Ehf{\difi_k} {(\adegree_{\difi_k})}}^{\gcd(\nn+1,\difi_k)}(t)$. By Proposition \ref{averaging}, for any $k\in\{0,\cdots,\ndifi\}$,
\[
[\difi_k]_{\Sh^{\nn+1}}^{\adegree_{\difi_k}+1}g_k(\Sh^{\nn+1}) \Ehf{\difi_k}{(\adegree_{\difi_k})}(t) = [\difi_k]_{\overline{\Sh}^{\nn+1}}^{\adegree_{\difi_k}+1} g_k(\overline{\Sh}^{\nn+1}) \widetilde{\Ehf{\difi_k} {(\adegree_{\difi_k})}}^{\gcd(\nn+1,\Ehp)}(t).
\]
Therefore, by Lemma \ref{tilde_linear} and Theorem \ref{Lam-Postnikov},
\[
\begin{split}
\chi_{quasi}(\mathcal{A}_{\Phi}^{[1,\nn]},t)&=\R_{\Phi}(\Sh^{\nn+1})\Eh_{\Phi}(t)\\
&=\R_{\Phi}(\Sh^{\nn+1})(\sum_{k\in \{\difi_0,\cdots,\difi_{\ndifi}\}}\Ehf{k}{(\adegree_{k})}(t))\\
&=(\C{\Sh^{\nn+1}}) \A_{\degree}(\Sh^{\nn+1}) (\sum_{k\in \{\difi_0,\cdots,\difi_{\ndifi}\}}\Ehf{k}{(\adegree_{k})}(t))\\
&=[\difi_0]_{\Sh^{\nn+1}}^{\adegree_{\difi_0}+1}g_0(\Sh^{\nn+1})\Ehf{\difi_0}{(\adegree_{\difi_0})}(t)+\cdots+
[\difi_{\adegree}]_{\Sh^{\nn+1}}^{\adegree_{\difi_{\ndifi}+1}}g_{\ndifi}(\Sh^{\nn+1})\Ehf{\difi_{\ndifi}}{(\adegree_{\difi_{\ndifi}})}(t)\\
&=[\difi_0]_{\overline{\Sh}^{\nn+1}}^{\adegree_{\difi_0}+1}g_0(\overline{\Sh}^{\nn+1})\widetilde{\Ehf{\difi_0}{(\adegree_{\difi_0})}}^{\gcd(\nn+1,\Ehp)}(t)+\cdots+
[\difi_{\ndifi}]_{\overline{\Sh}^{\nn+1}}^{\adegree_{\difi_{\ndifi}+1}}g_{\ndifi}(\overline{\Sh}^{\nn+1})
\widetilde{\Ehf{\difi_{\ndifi}}{(\adegree_{\difi_{\ndifi}})}}^{\gcd(\nn+1,\Ehp)}(t)\\
&=(\C{\overline{\Sh}^{\nn+1}})\A_{\degree}(\overline{\Sh}^{\nn+1})(\sum_{k\in \{\difi_0,\cdots,\difi_{\ndifi}\}}\widetilde{\Ehf{k}{(\adegree_{k})}}^{\gcd(\nn+1,\Ehp)}(t))\\
&=\R_{\Phi}(\overline{\Sh}^{\nn+1})(\sum_{k\in \{\difi_0,\cdots,\difi_{\ndifi}\}}\widetilde{\Ehf{k}{(\adegree_{k})}}^{\gcd(\nn+1,\Ehp)}(t))\\
&=\R_{\Phi}(\overline{\Sh}^{\nn+1})\avEh{\Phi}{\gcd(\nn+1,\Ehp)}(t).
\end{split}
\]
The characteristic quasi-polynomial $\chi_{quasi}(\mathcal{A}_{\Phi}^{[1,\nn]},t)$ has the period $\gcd(\nn+1,\Ehp)$ because $\avEh{\Phi}{\gcd(\nn+1,\Ehp)}$ has the period $\gcd(\nn+1,\Ehp)$.
\end{proof}
Note that $\avEh{\Phi}{1}(t)$ is a polynomial. The following comes immediately from Theorems \ref{g-Eulerian} and \ref{main theorem}.
\begin{corollary}\label{gene_worpitzky_pol}
\begin{equation}
t^{\ell}=\R_{\Phi}(\Sh)\tilde{\Eh}^1_{\Phi}(t).
\end{equation}
\end{corollary}
\begin{remark}
\item Proposition \ref{root Ehrhart} can also be proved using Corollary \ref{gene_worpitzky_pol}. First, note that if a function $f(t)$ is a polynomial, then $(\Sh f)(t)=(\overline{\Sh}f)(t)$. By Corollary \ref{gene_worpitzky_pol} and Lemma \ref{Lam-Postnikov}, 
\[
\begin{split}
\A_{\degree}(\overline{\Sh})\Eh_{A_{\degree}}(t)&=\R_{\Phi}(\overline{\Sh})\avEh{\Phi}{1}(t)\\
&=\C{\overline{\Sh}}\A_{\degree}(\overline{\Sh})\avEh{\Phi}{1}(t)\\
&=\A_{\degree}(\overline{\Sh})\C{\overline{\Sh}}\avEh{\Phi}{1}(t)\\
&=\A_{\degree}(\overline{\Sh})\C{\Sh}\Eh_{\Phi}(t).
\end{split}
\]
Thus, 
\[
\A_{\degree}(\Sh)(\Eh_{A_{\degree}}(t)-\C{\Sh}\Eh_{\Phi}(t))=0.
\]
If ($\Eh_{A_{\degree}}(t)-\C{\Sh}\Eh_{\Phi}(t))\neq 0$, then Lemma \ref{Shift congruences} implies that $(1-\Sh)$ divides $\A_{\degree}(\Sh)$, but $(1-\Sh)$ does not divide $\A_{\degree}(\Sh)$. Hence, 
\[
\Eh_{A_{\degree}}(t)-\C{\Sh}\Eh_{\Phi}(t)=0.
\]
\end{remark}

\begin{theorem}\label{corollary_1}
Let $\m:=\frac{\nn+1}{\gcd(\nn+1,\Ehp)}$. Then, 
\begin{equation}
\chi_{quasi}(\mathcal{A}_{\Phi}^{[1,\nn]},t)=(\prod_{j=0}^{\degree}\frac{1}{\m}[\m]_{\Sh^{\coxc_j\cdot \gcd(\nn+1,\Ehp)}}) \chi _{quasi}(\mathcal{A}_{\Phi}^{[1,\gcd(\nn+1,\Ehp)-1]},t).
\end{equation}
\end{theorem}
\begin{proof}
By Theorem \ref{main theorem}, Lemma \ref{S_bar_linear}, and Proposition \ref{g_Eulerian_cong},
\[
\begin{split}
\chi _{quasi}(\mathcal{A}_{\Phi}^{[1,\nn]},t)
&=\R_{\Phi}(\overline{\Sh}^{\nn+1})\avEh{\Phi}{\gcd(\nn+1,\Ehp)}(t)\\
&= \frac{1}{\m^{\degree+1}} ([\m]_{\overline{\Sh}^{\coxc_0 \gcd(\nn+1,\Ehp)}}\cdots
[\m]_{\overline{\Sh}^{\coxc_{\degree} \gcd(\nn+1,\Ehp)}})\R_{\Phi}(\overline{\Sh}^{\gcd(\nn+1,\Ehp)})\avEh{\Phi}{\gcd(\nn+1,\Ehp)}(t)\\
&=\frac{1}{\m^{\degree+1}} ([\m]_{\Sh^{\coxc_0 \gcd(\nn+1,\Ehp)}}\cdots
[\m]_{\Sh^{\coxc_{\degree} \gcd(\nn+1,\Ehp)}}) \chi _{quasi}(\mathcal{A}_{\Phi}^{[1, \gcd(\nn+1,\Ehp)-1]},t).\\
\end{split}
\]
\end{proof}
We prove Conjecture \ref{Postnikov-Stanley_2} using the following lemma, as used in \cite{Athanasiadis}, \cite{Postnikov-Stanley}, and \cite{Yoshinaga_1}.
\begin{lemma}\label{Postnikov-Stanley's lemma}(Postnikov--Stanley \cite{Postnikov-Stanley}, Lemma 9.13)
Let $f(t) \in \mathbb{C}[t]$. Suppose that all the roots of the equation $f(t)=0$ have real parts that are equal to $a$. Let $g(\Sh) \in \mathbb{C}[\Sh]$ be a polynomial such that every root of the equation $g(z)=0$ satisfies $|z|=1$. Then, all roots of the equation $g(\Sh)f(t)=0$ have real parts that are equal to $a+\frac{\mathrm{deg} g}{2}$.
\end{lemma}

\begin{theorem}\label{gcd_prime}
Let $\nn$ be an integer with $\gcd(\nn+1,\Ehp)=1$. Then,
\begin{equation}
\chi_{quasi}(\mathcal{A}_{\Phi}^{[1,\nn]},t)=(\prod_{j=0}^{\degree}\frac{1}{\nn+1}[\nn+1]_{\Sh^{\coxc_j}})t^{\degree}.
\end{equation}
In particular, the characteristic quasi-polynomial becomes a polynomial and any root $z$ of the equation  $\chi(\mathcal{A}_{\Phi}^{[1,\nn]},t)=0$ satisfies $\Re z=\frac{\nn \coxn}{2}$.
\end{theorem}
\begin{proof}
We calculate the characteristic polynomial using Theorems \ref{corollary_1} and \ref{g-Eulerian}.
\[
\begin{split}
\chi_{quasi}(\mathcal{A}_{\Phi}^{[1,\nn]},t)&=(\prod_{j=0}^{\degree}\frac{1}{\nn+1}[\nn+1]_{\Sh^{\coxc_j}}) \chi _{quasi}(\mathcal{A}_{\Phi}^{[1,0]},t)\\
&=(\frac{1}{\nn+1})^{\degree+1} ([\nn+1]_{\Sh^{\coxc_0}}[\nn+1]_{\Sh^{\coxc_1}} \cdots
[\nn+1]_{\Sh^{\coxc_{\degree}}})\R_{\Phi}(\Sh)\Eh_{\Phi}(t)\\
&=(\frac{1}{\nn+1})^{\degree+1} [\nn+1]_{\Sh^{\coxc_0}}[\nn+1]_{\Sh^{\coxc_1}} \cdots
[\nn+1]_{\Sh^{\coxc_{\degree}}}t^{\degree}.
\end{split}
\]
By Lemma \ref{Postnikov-Stanley's lemma}, the real part of any root of the equation $\chi(\mathcal{A}_{\Phi}^{[1,n]},t)=0$ is $\frac{n(c_0+c_1+\cdots+c_{\ell})}{2}=\frac{nh}{2}$.
\end{proof}
	
\begin{remark}
Theorem \ref{gcd_prime} is a generalization of the expression of the characteristic polynomial of $\mathcal{A}_{A_{\degree}}^{[1,\nn]}$ given by Postnikov and Stanley  \cite{Postnikov-Stanley} and the expression of $\mathcal{A}_{B_{\degree}}^{[1,\nn]}$, $\mathcal{A}_{C_{\degree}}^{[1,\nn]}$, and $\mathcal{A}_{D_{\degree}}^{[1,\nn]}$ for even values of $\nn$ given by Athanasiadis \cite{Athanasiadis}.
\end{remark}
\begin{example}[case $E_6$]
Let $\nn$ be a positive integer. Let $\m:=\frac{\nn+1}{\gcd(\nn+1,\Ehp)}$.
\[
\begin{split}
\text{If }\gcd(\nn+1,6)=1,\text{ then} \\
\chi_{quasi}(\mathcal{A}_{E_6}^{[1,\nn]},t)&=(\frac{1}{\m}[\m]_{\Sh})^{3} (\frac{1}{\m}[\m]_{\Sh^2})^{3} (\frac{1}{\m}[\m]_{\Sh^3})t^{6}.\\
\text{If }\gcd(\nn+1,6)=2,\text{ then} \\
\chi_{quasi}(\mathcal{A}_{E_6}^{[1,\nn]},t)&=(\frac{1}{\m}[\m]_{\Sh^{2}})^{3} (\frac{1}{\m}[\m]_{\Sh^4})^{3} (\frac{1}{\m}[\m]_{\Sh^6})\chi_{quasi}(\mathcal{A}_{E_6}^{[1,1]},t).\\
\text{If }\gcd(\nn+1,6)=3,\text{ then} \\
\chi_{quasi}(\mathcal{A}_{E_6}^{[1,\nn]},t)&=(\frac{1}{\m}[\m]_{\Sh^3})^{3} (\frac{1}{\m}[\m]_{\Sh^6})^{3} (\frac{1}{\m}[\m]_{\Sh^9})\chi_{quasi}(\mathcal{A}_{E_6}^{[1,2]},t).\\
\text{If }\gcd(\nn+1,6)=6,\text{ then} \\
\chi_{quasi}(\mathcal{A}_{E_6}^{[1,\nn]},t)&=(\frac{1}{\m}[\m]_{\Sh^6})^{3} (\frac{1}{\m}[\m]_{\Sh^{12}})^{3} (\frac{1}{\m}[\m]_{\Sh^{18}})\chi_{quasi}(\mathcal{A}_{E_6}^{[1,5]},t).
\end{split}
\]
\end{example}

\begin{theorem}\label{Ch_rad}
Let $\eta:=\frac{\gcd(\nn+1,\Ehp)}{\gcd(\nn+1,\rad(\Ehp))}$.
\begin{equation}
\chi(\mathcal{A}_{\Phi}^{[1,\gcd(\nn+1,\Ehp)-1]},t)=(\prod_{j=0}^{\degree}\frac{1}{\eta}[\eta]_{\Sh^{\coxc_j\cdot \gcd(\nn+1,\rad(\Ehp))}}) \chi(\mathcal{A}_{\Phi}^{[1,\gcd(\nn+1,\rad(\Ehp))-1]},t).
\end{equation}
\end{theorem}
\begin{proof}
We set $t\equiv 1 \bmod \Ehp$. Then, we have that $\avEh{\Phi}{\gcd(\nn+1,\Ehp)}(t)= \avEh{\Phi}{\gcd(\nn+1,\rad(\Ehp))}(t)$ by Proposition \ref{constituent_inv} and Proposition \ref{gene_Eh}. Hence, by Lemma \ref{S_bar_linear}, Proposition \ref{g_Eulerian_cong}, and Theorem \ref{main theorem},
\[
\begin{split}
\chi(\mathcal{A}_{\Phi}^{[1,\gcd(\nn+1,\Ehp)-1]},t)&=\R_{\Phi}(\overline
{\Sh}^{\gcd(\nn+1,\Ehp)})\avEh{\Phi}{\gcd(\nn+1,\Ehp)}(t)\\
&=\R_{\Phi}(\overline
{\Sh}^{\eta \gcd(\nn+1,\rad(\Ehp))})\avEh{\Phi}{\gcd(\nn+1,\rad(\Ehp))}(t)\\
&=(\prod_{j=0}^{\degree}\frac{1}{\eta}[\eta]_{\overline{\Sh}^{\coxc_j\cdot \gcd(\nn+1,\rad(\Ehp))}})\R_{\Phi}(\overline{\Sh}^{\gcd(\nn+1,\rad(\Ehp))})\avEh{\Phi}{\gcd(\nn+1,\rad(\Ehp))}(t)\\
&=(\prod_{j=0}^{\degree}\frac{1}{\eta}[\eta]_{\Sh^{\coxc_j\cdot \gcd(\nn+1,\rad(\Ehp))}}) \chi(\mathcal{A}_{\Phi}^{[1,\gcd(\nn+1,\rad(\Ehp))-1]},t).
\end{split}
\]	
\end{proof}

We now provide examples of Theorem \ref{Ch_rad} for $E_8$. Using the notation of Theorem \ref{Ch_rad}, in the case of $E_8$, $\eta$ can only take a value of $1$ or $2$. If $\eta=1$, then Theorem \ref{Ch_rad} is trivial. The following formulas are examples of Theorem \ref{Ch_rad} for $\eta=2$.
\begin{example}[Case $E_8$] 
\begin{equation}
\chi(\mathcal{A}_{E_8}^{[1,4-1]},t)=(\prod_{i=0}^{\degree}\frac{1}{2}[2]_{\Sh^{2c_{i}}})\chi(\mathcal{A}_{E_8}^{[1,2-1]},t).
\end{equation}
\begin{equation}
\chi(\mathcal{A}_{E_8}^{[1,12-1]},t)=(\prod_{i=0}^{\degree}\frac{1}{2}[2]_{\Sh^{6c_{i}}})\chi(\mathcal{A}_{E_8}^{[1,6-1]},t).
\end{equation}
\begin{equation}
\chi(\mathcal{A}_{E_8}^{[1,20-1]},t)=(\prod_{i=0}^{\degree}\frac{1}{2}[2]_{\Sh^{10c_{i}}})\chi(\mathcal{A}_{E_8}^{[1,10-1]},t).
\end{equation}
\begin{equation}
\chi(\mathcal{A}_{E_8}^{[1,60-1]},t)=(\prod_{i=0}^{\degree}\frac{1}{2}[2]_{\Sh^{30c_{i}}})\chi(\mathcal{A}_{E_8}^{[1,30-1]},t).
\end{equation}
\end{example}

\subsection{Verification of the Postnikov--Stanley Linial arrangement conjecture}\label{section:main_check}
We verify Conjecture \ref{Postnikov-Stanley_2} for $\Phi=E_6,E_7,E_8$, or $F_4$. We use the notation of Theorems \ref{corollary_1} and \ref{Ch_rad}. Recall that, according to these theorems, the following formula holds.
\begin{equation}
\chi(\mathcal{A}_{\Phi}^{[1,\nn]},t)=(\prod_{j=0}^{\degree}\frac{1}{\m}[\m]_{\Sh^{\coxc_j\cdot \gcd(\nn+1,\Ehp)}}) (\prod_{j=0}^{\degree}\frac{1}{\eta}[\eta]_{\Sh^{\coxc_j\cdot \gcd(\nn+1,\rad(\Ehp))}}) \chi(\mathcal{A}_{\Phi}^{[1,\gcd(\nn+1,\rad(\Ehp))-1]},t).
\end{equation}

If the real part of any root of the equation 
\[
\chi(\mathcal{A}_{\Phi}^{[1,\gcd(\nn+1,\rad(\Ehp))-1]},t)=0
\]
is $\frac{(\gcd(\nn+1,\rad(\Ehp))-1)\coxn}{2}$ for $\Phi \in \{E_6,E_7,E_8,F_4\}$, then Lemma \ref{Postnikov-Stanley's lemma} implies that Conjecture \ref{Postnikov-Stanley_2} is true. We have computed the characteristic polynomial such that the parameter $\nn+1$ is a factor of $\rad(\Ehp)$ other than $1$ and have determined the real part of the roots using a computational method. We use Theorem \ref{main theorem} and the calculation results of the Ehrhart quasi-polynomial given by Suter \cite{Suter} to compute the characteristic polynomial. The case of $\gcd(\nn+1,\rad(\Ehp))=\rad(\Ehp)$ has already been verified in \cite{Yoshinaga_1}. We present the characteristic polynomials for $\Phi \in \{E_6,E_7,E_8,F_4\} $ in the following tables.

\begin{landscape}
\begin{table}[htbp]
\centering
\caption{Characteristic polynomials for $E_6$ ($\rad(\Ehp)=6$).}
{\normalsize 
\begin{tabular}{|l|c|c|}
\hline
$\chi(\mathcal{A}_{E_6}^{[1,n]},t)$
&$n$&real part\\
\hline\hline
$t^6-36t^5+630t^4-6480t^3+40185t^2-140076t+211992$&$2-1$&$6$\\
\hline
$t^6-72t^5+2400t^4-46080t^3+528600t^2-3396672t+9474200$&$3-1$&$12$\\
\hline
$t^6-180t^5+14550t^4-666000t^3+18019065t^2-271143900t+1762474040$&$6-1$&$30$\\
\hline
\end{tabular}
}
\label{fig:E_6}
\end{table}

\begin{table}[htbp]
\centering
\caption{Characteristic polynomials for $F_4$ ($\rad(\Ehp)=6$).}
{\normalsize 
\begin{tabular}{|l|c|c|}
\hline
$\chi(\mathcal{A}_{F_4}^{[1,\nn]},t)$
&$\nn$&real part\\
\hline\hline
$t^4-24t^3+258t^2-1368t+2917$&$2-1$&$6$\\
\hline
$t^4-48t^3+1000t^2-10176t+41572$&$3-1$&$12$\\
\hline
$t^4-120t^3+5986t^2-143160t+1361989$&$6-1$&$30$\\
\hline
\end{tabular}
}
\label{fig:F_4}
\end{table}

\begin{table}[htbp]
\centering
\caption{Characteristic polynomials for $E_7$ ($\rad(\Ehp)=6$).}
{\normalsize 
\begin{tabular}{|l|c|c|}
\hline
$\chi(\mathcal{A}_{E_7}^{[1,\nn]},t)$
&$\nn$&real part\\
\hline\hline
$t^7-63t^6+1953t^5-36855t^4+446355t^3-3417309t^2+15154251t-29798253$&$2-1$&$9$\\
\hline
$t^7-126t^6+7476t^5-264600t^4+5948040t^3-84088368t^2+687202712t-2490427440$&$3-1$&$18$\\
\hline
$t^7-315t^6+45465t^5-3850875t^4+204937635t^3 -6808068225t^2+130052291075t-1097517119625$&$6-1$&$45$\\
\hline
\end{tabular}
}
\label{fig:E_7}
\end{table}

\begin{table}[htbp]
\centering
\caption{Characteristic polynomials for $E_8$ ($\rad(\Ehp)=30$).}
{\scriptsize 
\begin{tabular}{|l|c|c|}
\hline
$\chi(\mathcal{A}_{E_8}^{[1,\nn]},t)$
&$\nn$&real part\\
\hline\hline
$t^8-120t^7+7140t^6-264600t^5+6540030t^4-108901800t^3+1181603220t^2-7583286600t+21918282249$&$2-1$&$15$\\
\hline
$t^8-240t^7+27440t^6-1915200t^5+88161360t^4-2716963200t^3+54385106720t^2-643164643200t+3426392186728$&$3-1$&$30$\\
\hline
$t^8-480t^7+107520t^6-14515200t^5+1281219408t^4-75249457920t^3+2857900896480t^2-63918602553600t+642465923287416$&$5-1$&$60$\\
\hline
$t^8-600t^7+167300t^6-28035000t^5+3065453790t^4-222698637000t^3+10449830016500t^2-288505461225000t+3577184806486057$&$6-1$&$75$\\
\hline
$t^8-1080t^7+538020t^6-160234200t^5+31018986558t^4-3977954041320t^3+328758988903380t^2-15957853314798600t+347373804233610441$&$10-1$&$135$\\
\hline
$t^8-1680t^7+1297520t^6-597643200t^5+178602069408t^4-35307879102720t^3+4493170619530880t^2-335521093135065600t+11227745283721390816$&$15-1$&$210$\\
\hline
$t^8-3480t^7+5550020t^6-5266510200t^5+3236633286558t^4-1314003597910920t^3+343011765319289780t^2-52494228716611434600t+3597446896074261934441$&$30-1$&$435$\\
\hline
\end{tabular}
}
\label{fig:E_8}
\end{table}
\end{landscape}

\medskip

\noindent
{\bf Acknowledgements.}
I am very grateful to Masahiko Yoshinaga for his various comments on how to improve this paper, for many discussions on the content of this paper, and for his suggestions for addressing the Postnikov--Stanley Linial arrangement conjecture. I thank Stuart Jenkinson, PhD, from Edanz Group (https://en-author-services.edanzgroup.com/ac) for editing a draft of this manuscript. The author also thanks the Department of Mathematics, Hokkaido University 
and JSPS KAKENHI JP18H01115 (PI: M. Yoshinaga) for financial supports.


\begin{thebibliography}{16}
\bibitem{Athanasiadis_0}
C.A. Athanasiadis, Algebraic combinatorics of graph spectra, subspace arrangements and Tutte polynomials, 
Ph.D. thesis, M.I.T. (1996).
\bibitem{Athanasiadis} 
C.A. Athanasiadis, 
Extended Linial hyperplane arrangements for root systems and a
conjecture of Postnikov and Stanley,  
J. Algebraic Combin. {\bf 10} (1999), no. 3, 207-225. 
\bibitem{Beck-Robinson} 
M. Beck, S. Robins, 
Computing the continuous discretely. Integer-point enumeration in polyhedra, 
Undergraduate Texts in Mathematics,
Springer, New York, 2007, xviii+226 pp. 
\bibitem{Humphreys}
J.E. Humphreys, 
Reflection groups and Coxeter groups, 
Cambridge Studies in Advanced Mathematics, 29, Cambridge University Press, Cambridge, 1990, xii+204 pp.
\bibitem{Iijima-Sasaki-Takahashi-Yoshinaga} K. Iijima, K. Sasaki, Y. Takahashi, M. Yoshinaga,  Eulerian polynomials and polynomial congruences, 
Contrib. Discret. Math., {\bf 14} (2019), no. 1, 46-54.
\bibitem{Kamiya-Takemura-Terao_1} 
H. Kamiya, A. Takemura, H. Terao, 
Periodicity of hyperplane arrangements with integral coefficients 
modulo positive integers,
J. Algebraic Combin. {\bf 27} (2008), no. 3, 317-330. 
\bibitem{Kamiya-Takemura-Terao_0}
H. Kamiya, A. Takemura, H. Terao, 
The characteristic quasi-polynomials of the arrangements of root systems and mid-hyperplane arrangements, in: F. El Zein, A.I. Suciu, M. Tosun, A.M. Uluda\v{g}, S. Yuzvinsky (Eds.), Arrangements, local systems and singularities, 
Progr. Math. {\bf 283}, Birkh\"auser Verlag, Basel, 2010, pp. 177-190. 
\bibitem{Kamiya-Takemura-Terao_2}
H. Kamiya, A. Takemura, H. Terao, 
Periodicity of non-central integral arrangements modulo positive integers, 
Ann. Comb. {\bf 15} (2011), no. 3,  449-464. 
\bibitem{Lam-Postnikov}
T. Lam, A. Postnikov, 
Alcoved polytopes II, in: V.G. Kac, V.L. Popov (Eds.), Lie Groups, Geometry, and Representation Theory, Progr. Math., {\bf 326}, Birkh\"auser Basel, 2018, pp. 253-272.
\bibitem{Orlik-Terao}
P. Orlik, H. Terao, 
Arrangements of hyperplanes, 
Grundlehren der Mathematischen Wissenschaften, 300, Springer-Verlag Berlin Heidelberg, 1992, xviii+325 pp.
\bibitem{Postnikov-Stanley}
A. Postnikov, R. Stanley, 
Deformations of Coxeter hyperplane arrangements,
J. Comb. Theory Ser. A, {\bf 91} (2000), no. 1-2, 544-597. 
\bibitem{Stanley-EC1}
R. Stanley, 
Enumerative Combinatorics. Volume 1, second ed., Cambridge Studies in Advanced Mathematics, vol. 49, Cambridge University Press, Cambridge, 2012, xiv+626 pp.
\bibitem{Suter}
R. Suter, 
The number of lattice points in alcoves and the exponents of the 
finite Weyl groups, 
Math. Comp., {\bf 67} (1998), no. 222, 751-758. 
\bibitem{Worpitzky}
J. Worpitzky, 
Studien \"uber die Bernoullischen und Eulerischen Zahlen, 
J. Reine Angew. Math., {\bf 94} (1883) 203-232.
\bibitem{Yoshinaga_1}
M. Yoshinaga, 
Worpitzky partitions for root systems and characteristic quasi-polynomials, 
Tohoku Math. J., {\bf 70} (2018), no. 1, 39-63.
\bibitem{Yoshinaga_2} M. Yoshinaga, Characteristic polynomials of Linial arrangements for exceptional root systems, J. Comb. Theory Ser. A, {\bf 157} (2018) 267-286.
\end{thebibliography}
\end{document}